\documentclass[a4paper,11pt,fleqn]{article}
\usepackage{amsmath,amssymb,amsthm}
\usepackage{mathrsfs} 
\usepackage{euscript}
\usepackage{color}
\usepackage[shortlabels]{enumitem}
\usepackage{charter}
\usepackage{array}
\usepackage{graphicx}
\usepackage[normalem]{ulem}

\definecolor{labelkey}{rgb}{0,0.08,0.45}
\definecolor{refkey}{rgb}{0,0.6,0.0}
\definecolor{Brown}{rgb}{0.45,0.0,0.05}
\definecolor{dgreen}{rgb}{0.00,0.49,0.00}
\definecolor{dblue}{rgb}{0,0.08,0.75}
\RequirePackage[colorlinks,hyperindex]{hyperref} 
\hypersetup{linktocpage=true,citecolor=dblue,linkcolor=dgreen}

\newtheorem{theorem}{Theorem}[section]

\newtheorem{lemma}[theorem]{Lemma}
\newtheorem{fact}[theorem]{Fact}
\newtheorem{proposition}[theorem]{Proposition}
\newtheorem{problem}[theorem]{Problem}

\theoremstyle{definition}

\newtheorem{algorithm}[theorem]{Algorithm}
\newtheorem{remark}[theorem]{Remark}

\numberwithin{equation}{section}

\DeclareMathOperator*{\esssup}{\ensuremath{\text{\rm ess\,sup}}}
\DeclareMathOperator*{\essinf}{\ensuremath{\text{\rm ess\,inf}}}
\DeclareMathOperator*{\dom}{\ensuremath{\text{\rm dom}}}
\DeclareMathOperator*{\argmin}{\ensuremath{\text{\rm arg\,min}}}
\DeclareMathOperator*{\argmax}{\ensuremath{\text{\rm arg\,max}}}
\DeclareMathOperator*{\range}{\ensuremath{\text{\rm Im}}}
\DeclareMathOperator*{\inte}{\ensuremath{\text{\rm int}}}
\DeclareMathOperator*{\cl}{\ensuremath{\text{\rm cl}}}

\DeclareMathOperator{\Ker}{\ensuremath{\text{\rm Ker}}}

\DeclareMathOperator{\Diag}{\ensuremath{\text{\rm diag}}}

\newcommand{\EE}{\ensuremath{\mathbb E}}
\newcommand{\PP}{\ensuremath{\mathbb P}}
\newcommand{\VV}{\ensuremath{\mathbb V}}

\newcommand{\hatC}{C}
\newcommand{\Lc}{C}

\newcommand{\R}{\mathbb R}
\newcommand{\extR}{\left]-\infty,+\infty\right]}
\newcommand{\C}{\mathcal C}
\newcommand{\Hh}{\ensuremath X}
\newcommand{\Gg}{\ensuremath Y}
\newcommand{\N}{\mathbb N}

\newcommand{\Leg}{\phi}
\newcommand{\Bre}{D_\Leg}
\newcommand{\Bredual}{D_{\Leg^*}}
\newcommand{\Bdist}{D}
\newcommand{\Pc}{P}
\newcommand{\intdom}{\inte(\dom \phi)}

\providecommand{\norm}[1]{\lVert#1\rVert}
\providecommand{\scalarp}[1]{\langle#1\rangle}
\providecommand{\abs}[1]{\lvert#1\rvert}

\title{ { \sffamily The method of Bregman projections in 
deterministic and  stochastic convex feasibility problems} } 

\author{Vladimir Kostic\thanks{Istituto Italiano di Tecnologia, Via Melen, 83,  
		16152 Genova, Italy 
		({\tt vladimir.kostic@iit.it}) and Department of Mathematics and Informatics, Faculty of Science, University of Novi Sad, Trg Dositeja Obradovića 4, 21000 Novi Sad, Serbia.} \and Saverio Salzo\thanks{Istituto Italiano di Tecnologia, Via Melen, 83,  
        16152 Genova, Italy ({\tt saverio.salzo@iit.it}).}}
\date{}

\begin{document}
\maketitle

\begin{abstract}
In this work we study the method of Bregman projections for deterministic and stochastic convex feasibility problems
with three types of control sequences for the selection of sets during the algorithmic procedure:  greedy, random, and
adaptive random. We analyze in depth the case of affine feasibility problems showing that the iterates generated by 
the proposed methods converge Q-linearly and providing also explicit global and local rates of convergence. This work 
generalizes from one hand recent developments in randomized methods for the solution of linear systems based on
orthogonal projection methods. On the other hand, our results yield global and local Q-linear rates of convergence for 
the Sinkhorn and Greenhorn algorithms in discrete entropic-regularized optimal transport, for the first time, even in the 
multimarginal setting.
\end{abstract}

 \vspace{1ex}

\section{Introduction}
The convex feasibility problem consists in finding a point in the intersection of a finite family of closed convex sets. 
Such problem arises in several areas of mathematics and applied sciences, such as best approximation theory 
\cite{DEU1991} and image reconstruction \cite{COM1993}. A classical approach to solve that problem is the method
of \emph{cyclic orthogonal projections} \cite{HAL1962}, which generates a sequence of points by projecting onto each
constraint set cyclically. In 1967 Bregman, in \cite{BRE1967}, generalized this method by allowing non-orthogonal 
projections which are constructed as follows. Given a well-behaved convex function $\Leg$, construct the so called 
\emph{Bregman distance}\footnote{Note that is not a distance in the sense of metric topology and even when 
$\phi(x) = (1/2) \norm{x}^2$ it is one half of the square of the distance between $x$ and $y$.} (or \emph{divergence})
between two points $x$ and $y$ as
\begin{equation}
\Bre(x,y) = \Leg(x) - \Leg(y) - \scalarp{x-y, \nabla \Leg(x)}
\end{equation}
and then define the Bregman projection of a point $x$ onto a given closed convex set as the point in the set which 
minimizes the Bregman distance from $x$. He studied two versions of the algorithm depending on the strategy 
(\emph{control}) for calling the various projections during the algorithmic procedure: \emph{cyclic} and \emph{the 
most remote} set control. Since then, a number of studies followed, extending the theory to more general set controls:
\emph{almost cyclic} and \emph{repetitive} \cite{CEN1981,CEN1996,COM1993}. The Bregman projection method, 
with respect to the Kullback-Leibler divergence, has found numerous applications in optimal transport \cite{BEN2015} 
and probability and statistics, where it is known as the \emph{iterative proportional fitting procedure} (IPFP)
\cite{RUS1995}.

In this work, we focus on the \emph{convex feasibility problem} in an Euclidean space $X$, addressing both the
\emph{deterministic} and the \emph{stochastic} variants (possibly with an infinite uncountable number of sets)
as described in the following problems.
\begin{problem}[Deterministic]
\label{P1}
Let $\mathcal{C} \!=\! (C_i)_{i \in I}$ be a family of nonempty closed convex sets in $X$.
Let $C = \bigcap_{i \in I} C_i$ and suppose that $C \neq \varnothing.$ Find $x \in C$.
\end{problem}

\begin{problem}[Stochastic]
\label{P2}
Let $(I, \mathcal{I})$ be a measurable space, $\mathcal{C} = (C_i)_{i \in I}$ a family of nonempty closed convex 
sets in $X$, such that $i \mapsto C_i$ is a measurable set-valued mapping, and let $\xi$ be an $I$-valued random
variable. Let $C = \{ x \in X \,\vert\, x \in C_\xi, \PP\text{-a.s.} \}$ and suppose that $C \neq \varnothing$. Find $x \in C$.
\end{problem}

In dealing with the two problems above we consider the method of Bregman projections. The algorithm is detailed 
below, where we denote by $P_{C_i}$ the Bregman projection onto the set $C_i$ w.r.t.~$\Leg$.
\begin{algorithm}[The Bregman projection method]
\label{algoABP}
Let $x_0 \in \inte(\dom\Leg)$. Iterate
\vspace{-1ex}
\begin{equation}
\label{eq:algo1}
\begin{array}{l}
\text{for}\;k=0,1,\ldots\\
\left\lfloor
\begin{array}{l}
\text{choose } \xi_k \in I\\
x_{k+1} = P_{C_{\xi_k}}(x_{k}).
\end{array}
\right.
\end{array}
\end{equation}
\end{algorithm}
\noindent
The sequence $(\xi_k)_{k \in\N}$ is called the \emph{set control sequence} and, depending on which of the two 
problems above we consider, it can be deterministic or stochastic.

\subsection{Contribution}
In the following we summarize the main contribution of this paper. We denote by $D_{C}(x) = \inf_{z \in C} \Bre(z,x)$ 
the Bregman distance from $x$ to the set $C$. The following holds.
\begin{itemize}
\item As for Problem~\ref{P1}, if the sets $C_i$'s are affine and the $\xi_k$'s are chosen according to the greedy 
strategy $\xi_k \in \argmax_{i \in I} D_{C_i}(x_{k})$ (the most remote set control), we prove that 
$\Bre(P_C(x_0), x_k) =D_C(x_k) \to 0$ with Q-linear rate and we provide both global and local rates of convergence.
See Theorem~\ref{thm:main2}.
\item As for Problem~\ref{P2}, if the $\xi_k$'s are random variables which are independent copies of $\xi$, then we 
prove that the iterates converge almost surely and in mean square to a random variable taking values in the set $C$.
See Theorem~\ref{thm:conv2}. Moreover, if the sets $C_i$'s are affine and the $\xi_k$'s are random variables which 
are either independent copies of $\xi$ or with distribution adaptively depending on $x_{k}$, we prove that 
$\EE[\Bre(P_C(x_0), x_k)] =\EE[D_C(x_k)] \to 0$ with Q-linear rate and we provide both global and local rates of
convergence. See Theorem~\ref{thm:main_random}, Remark~\ref{rmk:Qlin}, and Theorem~\ref{thm:main_adaptive}.
\end{itemize}

Below we comment on the results. 1) To the best of our knowledge the Q-linear convergence of the Bregman 
projection method in (deterministic and stochastic) affine feasibility problems is new. This fully generalizes the 
well-known linear rate of convergence of the orthogonal projection method \cite{KAC1937,NRP2019,SV2009,RT2020}
to Bregman projections. We stress that this extension is truly nontrivial. Indeed, while in the classical setting one 
works with orthogonal projections w.r.t.~an Euclidean norm, in the Bregman setting this can be done only locally 
by approximating  the Bregman distance through the semi-norm induced by the Hessian (which indeed we allow to 
be possibly rank deficient) of the Fenchel conjugate of $\Leg$. 2) Concerning the stochastic feasibility problem in 
general, the Bregman projection approach and its converge is also new. 3) The results cover a number of interesting
(Legendre) functions $\Leg$: among others, Burg entropy, Boltzmann-Shannon entropy, Fermi-Dirac entropy, and 
$p$-norms with $1<p \leq 2$. 4) Novel applications include sketch \& Bregman project methods for solving linear 
systems of equations and generalization of Sinkhorn and Greennkhorn algorithms for regularized multimarginal 
optimal transport, to name a few. In particular, this analysis establishes global and local Q-linear convergence of 
the Greenkhorn algorithm and of the (stochastic and greedy) iterative KL projection algorithm for multimarginal 
entropic-regularized optimal transport.

\subsection{Related works}
Although the method of Bregman projections for convex feasibility problems has a long history and has proved to 
be at the basis of several important algorithms in science \cite{DT2007,PC2019}, unlike its Euclidean version 
(the method of orthogonal projections), it remained in the domain of deterministic intersection of sets, i.e., 
Problem \ref{P1}. Indeed, even if quite general deterministic set control sequences, called repetitive, are allowed
\cite{CEN1981,CEN1996}, they are meaningful and can possibly handle the stochastic setting only when the number 
of sets is finite (see~Remark~\ref{rmk:repcontrol}). On the other hand Problem \ref{P2}, even with an uncountable
number of sets, was first considered in \cite{BUT1995} and tackled via an \emph{expected orthogonal projection method},
which was later extended to Bregman projections in \cite{BUT2000}. However, these types of methods are different 
from Algorithm~\ref{algoABP}: indeed they are defined as $x_{k+1} = \EE[P_{C_{\xi_k}} (x_k)]$, so that they generate 
a non-stochastic sequence. Instead, the method of stochastic \emph{orthogonal} projections for Problem \ref{P2} was
introduced in \cite{NED2010} and in recent years has received a renewed attention also thanks to the fast development
of machine learning and data science \cite{HER2019,NRP2019, MT2020}.

The linear convergence of the method is known only in the deterministic case and for some special affine feasibility
problems. A prominent example is that of the  Sinkhorn algorithm for entropic optimal transport, which can be viewed 
as alternating Bregman projections (w.r.t.~the Kullback-Leibler divergence) between two suitable affine sets: global 
and local convergence rates are discussed in \cite{KNI2008,PC2019}. We also mention the work \cite{IUS1991}, by
Iusem, which shows \emph{local} linear convergence for a row-action method for general linear inequality constraints 
with an almost cyclic control sequence. The method is not purely alternating projections (as Algorithm~\ref{algoABP}), 
but is in fact  a primal-dual algorithm (originally introduced in \cite{BRE1967}) specifically designed for the case that the $C_i$'s 
are halfspaces. Besides, for this method no global linear rate is known.
		
Since affine feasibility problems are essentially linear systems of equations, in the following we also discuss 
connections with the field of numerical linear algebra. In the last decades, mainly triggered by problems in machine
learning, randomized linear solvers have emerged as a way to approach large linear systems. One of the main
achievements in this direction is the work by Strohmer and Vershynin \cite{SV2009} who studied a randomized 
version of the Kaczmarz method \cite{KAC1937}. More recently, in \cite{GR2015,GMMN2020,NRP2019,RT2020}, 
this method has been extended to the more general framework of \emph{sketch \& project}, which is designed to 
compute the minimal norm solution of a feasible linear system by making a sequence of orthogonal projections 
onto sketched systems of smaller dimensions. In \cite{GMMN2020} different strategies of choosing sketched systems 
at each iteration were considered: greedy, random, and adapive. For these strategies linear convergence rates were
obtained, which in the special case of the randomized Kaczmarz methods reduce to those in \cite{SV2009}. Our work
extends \cite{GMMN2020, RT2020, SV2009}, since we allow Bregman projections onto the sketched systems, 
proving linear convergence for all the types of sketching strategies described above. We call our method 
\emph{sketch \& Bregman project}. We stress that when our results specialize to orthogonal projections, they 
completely recover the ones in \cite{GMMN2020, RT2020, SV2009},  revealing that our analysis is indeed tight 
(see Section~\ref{sec:SketchProject}).

\subsection{Outline of the paper}
In the next section we provide notation and basic facts about the Bregman projections. Section~\ref{sec:BPM} 
analyzes Algorithm~\ref{algoABP} for the general case of convex sets. In Section~\ref{sec:BPMaff} we present the 
main results of this work, which concerns the linear convergence of the deterministic and stochastic Bregman 
projection method for the affine feasibility problem. Finally, Section~\ref{sec:Applications} discusses in more details 
some relevant applications, meaning, sketch \& Bregman project methods for solving linear systems and regularized
optimal transport problems.

\section{Preliminaries}
In this section we provide notation and basic concepts and results related to the Bregman projection method.

\subsection{Notation and basic background}
In this work $\Hh$ is an Euclidean space with scalar product $\scalarp{\cdot,\cdot}$ and induced norm $\norm{\cdot}$.
The interior of a set $C \subset \Hh$ is denoted by $\inte(C)$ and its boundary by $\mathrm{bdry}(C)$. We set 
$\R_+ = \left[0,+\infty\right[$ and $\R_{++} = \left]0,+\infty\right[$.

Let $\Leg\colon\Hh\to\extR$ be an extended-real valued function. The set of minimizers of the function $\Leg$
is denoted by $\argmin_{x \in \Hh} \Leg(x)$, the \emph{domain} of $\phi$  is 
$\dom \Leg :=\{x\in\Hh \,\vert\, \Leg(x)<+\infty\}$ and $\phi$ is \emph{proper} when $\dom \Leg \neq \varnothing$. 
The function $\Leg$ is \emph{convex} if $\Leg(t x+(1-t)y)\leq t\Leg(x) + (1-t) \Leg(y)$ for all $x,y\in \dom\Leg$ and
$t\in[0,1]$. If the above inequality is strict when $0<t<1$ and $x \neq y$, the function is \emph{strictly convex}. 
The function $\Leg$ is \emph{closed} if the sublevel sets $\{x\in\Hh\,\vert\, \Leg(x)\leq t\}$ are closed in $\Hh$ for 
any $t\in\R$. For a convex function $\Leg\colon\Hh\to\extR$, we denote by $\Leg^{*}$ its \emph{Fenchel conjugate}, 
that is, $\Leg^*\colon \Hh\to\extR$, $\Leg^{*}(y):=\sup_{x\in\Hh}\{\scalarp{x,y}-\Leg(x)\}$. The conjugate of a convex
function is always closed and convex, and if $\Leg$ is proper closed and convex, then $(\Leg^{*})^{*}=\Leg$. 

A proper closed and convex function $\Leg$ is \emph{essentially smooth} if it is differentiable on 
$\mathrm{int}(\dom \Leg) \neq \varnothing$, and $\norm{\nabla \Leg(x_n)} \to +\infty$ whenever 
$x_n \in \inte(\dom \Leg)$ and $x_n \to x \in \mathrm{bdry}(\dom \Leg)$.  The function $\Leg$ is 
\emph{essentially strictly convex} if $\inte(\dom \Leg^*)\neq \varnothing$ and is strictly convex on every 
convex subset of $\dom \partial \Leg$. A \emph{Legendre} function is a proper closed and convex function 
which is also essentially smooth and essentially strictly convex. A function is Legendre if and only if its conjugate is so.
Moreover, if $\Leg$ is a Legendre function, then $\nabla \Leg\colon \inte(\dom \Leg) \to \inte(\dom \Leg^*)$ and 
$\nabla \Leg^*\colon \inte(\dom \Leg^*) \to \inte(\dom \Leg)$ are bijective, inverses of each other, and continuous. 
See \cite[Sec.~26]{ROC1970}. Given a Legendre function $\Leg$, the \emph{Bregman distance} associated to 
$\Leg$ is the function $\Bre\colon\Hh \times \Hh \to [0,+\infty]$ such that
\begin{equation}
\label{Bdiv}
\Bre(x,y) =
\begin{cases}
\Leg(x)-\Leg(y)-\scalarp{x-y,\nabla \Leg(y)}&\text{if } y \in \inte(\dom \Leg)\\
+\infty &\text{otherwise}.
\end{cases}
\end{equation}

In the following we will use some important properties of Bregman distances generated by Legendre functions. 
See \cite{BAU2001,BAU1997,BAU2003}. Note that item \ref{prop:bregman_div_vii} follows from Taylor's formula 
for $\Leg$. 

\newpage
\begin{fact}\label{prop:bregman_div}
Let $\Leg$ be a Legendre function. Then the following properties hold.
\begin{enumerate}[label={\rm (\roman*)}]
\item\label{prop:bregman_div_0} 
$(\forall\,x \in \dom \Leg)(\forall\, y \in \inte(\dom \Leg))$\ 
$\Bre(x,y) = \Leg(x) + \Leg^*(\nabla \Leg(y)) - \scalarp{x,\nabla \Leg(y)}$.
\item $(\forall\, y\in \inte(\dom \Leg) )$\ $\Bre(\cdot,y)$ is a strictly convex on $\inte(\dom \Leg)$ and coercive.
\item\label{prop:bregman_div_iia} $(\forall\, x,y \in \inte(\dom \Leg))$\ $\Bre(x,y)=0\ \Leftrightarrow\ x=y$.
\item\label{prop:bregman_div_ii} $(\forall\, x,y \in \inte(\dom \Leg))$
\ $\Bre(x,y)+\Bre(y,x) = \scalarp{x-y, \nabla\Leg(x)-\nabla\Leg(y)}\geq0$.
\item\label{prop:bregman_div_iv} 
(Three-Point Identity \cite{CHE1993}) For every $x \in \Hh$ and $y,z\in\inte(\dom \Leg)$, we have 
\begin{equation}\label{3p}
\Bre(x,z) = \Bre(x,y)+\Bre(y,z)+ \scalarp{x- y, \nabla\Leg(y)-\nabla\Leg(z)}.
\end{equation}
\item\label{prop:bregman_div_vi} $(\forall\, x,y \in \inte(\dom \Leg))$\ $\Bre(x,y)=\Bredual(\nabla\Leg(y),\nabla\Leg(x))$.
\item\label{prop:bregman_div_iii} $\Bre$ is continuous on $\inte(\dom \Leg) \times \inte(\dom \Leg)$.
\item Suppose that $\Leg$ is twice differentiable on $\inte(\dom\Leg)$. Then
\begin{equation}
\big(\forall x\!\in\inte(\dom\Leg),\, \nabla^2\Leg(x)\text{ is invertible}\big)\Leftrightarrow
\big(\Leg^*\!\text{ is twice differentiable}\big).
\end{equation}
\item\label{prop:bregman_div_ix} Suppose that $\dom \Leg^*$ is open. Then, for every $x \in \inte(\dom \Leg)$,  
the sublevel sets of $\Bre(x,\cdot)$ are compact, and hence $\Bre(x,\cdot)$ is lower semicontinuous.
\item\label{prop:bregman_div_x} Suppose that $\dom \Leg^*$ is open. Then, for every $x \in \inte(\dom \Leg)$,  
and every sequence $(y_k)_{k \in \N}$ in $\inte(\dom \Leg)$
\begin{equation}
\Bre(x,y_k) \to 0\ \Rightarrow\  y_k \to x.
\end{equation}
Consequently, for every $x\in \inte(\dom\Leg)$ and $\varepsilon>0$, there exists $\delta>0$ such that for every 
$y\in \inte(\dom\Leg)$, $D_{\Leg}(x,y)< \delta \implies  \norm{x-y}<\varepsilon$.
\item\label{prop:bregman_div_vii} 
If $\Leg$ is twice differentiable on $\inte(\dom \Leg)$, then for every 
$x,y\in\inte(\dom \Leg)$ there exists $\xi\in[x,y]$ such that 
\begin{equation}\label{hess}
\Bre(x,y)=\frac{1}{2}\scalarp{\nabla^2\Leg(\xi)(x-y),x-y}.
\end{equation}
Moreover, for every $y \in \inte(\dom \Leg)$ and every $\varepsilon>0$ there exists 
$\delta>0$ such that, for every $x\in \inte(\dom \Leg)$ such that $x-y\not\in\Ker(\nabla^2\Leg(y))$,
\begin{equation}
\norm{x - y} \leq \delta\ \Rightarrow\ 
\bigg\lvert \frac{\Bre(x,y) - \frac 1 2 \scalarp{\nabla^2 \Leg(y) (x-y),x-y}}{\frac 1 2 \scalarp{\nabla^2 \Leg(y) (x - y), x - y}} \bigg\rvert
\leq \varepsilon.
\end{equation}
\end{enumerate}
\end{fact}

In addition to the above facts, we will use the following ones, too.
	
\begin{fact} \label{fact_inf}
Let $A\colon\Hh\to\Gg$ be a linear operator and let $A^\dag$ be its \emph{Moore-Penrose pseudoinverse}. 
Then $A A^\dag = A (A^* A)^\dag A^*$ is the orthogonal projector onto $\range(A)$, and 
$\norm{A^\dagger}^{-1} = \inf_{z\in\Ker(A)^{\perp}\setminus\{0\}} \norm{Az}/\norm{z}$ is the smallest positive 
singular value of $A$.
\end{fact}	

\begin{fact}[{\cite[Example~5.1.5]{DUR2010}}]
\label{f:20190112d}
Let $\zeta_1$ and $\zeta_2$ be independent random variables with values in the measurable spaces 
$\mathcal{Z}_1$ and $\mathcal{Z}_2$ respectively. Let $\varphi\colon \mathcal{Z}_1\times \mathcal{Z}_2 \to \R$ 
be measurable and suppose that $\EE[\abs{\varphi(\zeta_1,\zeta_2)}]<+\infty$. Then 
$\EE[\varphi(\zeta_1,\zeta_2) \,\vert\, \zeta_1] = \psi(\zeta_1)$, where for all $z_1 \in \mathcal{Z}_1$, 
$\psi(z_1) = \EE[\varphi(z_1, \zeta_2)]$.
\end{fact}

\begin{fact}[{\cite[Theorem~3.2.4]{DUR2010}}]
\label{fact_convRV}
Let $(x_k)_{k \in \N}$ be a sequence of $X$-valued random variable and let $x$ be an $X$-valued random variable. 
Then the following hold.
\begin{enumerate}[label={\rm (\roman*)}]
\item 
Suppose that $x_k$ are uniformly essentially bounded, i.e., $\sup_{k \in \N} \esssup \norm{x_k}<+\infty$. 
Then $x_k \to x$  $\PP$-a.s. $\Rightarrow$ $\EE[\norm{x_k-x}^2]\to0$. 
\item 
Suppose that $x_k \in U\subset X$ $\PP$-a.s.~and $T\colon U \to Y$ is continuous. Then $x_k \to x$ in distribution 
$\Rightarrow$ $T(x_k) \to T(x)$ in distribution.
\end{enumerate}
\end{fact}

\subsection{The Bregman projection onto a convex set}
In this section we recall the definition and the main properties of the Bregman projection operator.
We first address the general case of a convex set and then the special case of an affine set.

Let $C\subset\Hh$ be a nonempty closed and convex set. Let $\Leg$ be a Legendre function on $\Hh$ such 
that $C \cap \inte(\dom \Leg) \neq \varnothing$ and let $\Bre$ be its associated Bregman distance. 
Let $x \in \inte(\dom \Leg)$. Then, the following optimization problem 
\begin{equation}
\label{eq:BregProj}
\min_{z\in C}{\Bre(z,x}) 
\end{equation}
has a unique solution which is in $C\cap \inte(\dom \Leg)$ \cite[Corollary~7.9]{BAU2001}. In other words, 
the operator $\Pc_{C}\colon \inte(\dom \Leg)\to C\cap \inte(\dom \Leg)$ such that, for all $x \in \inte(\dom \Leg)$,
\begin{equation}
\label{eq:Breproj}
\Pc_{C}(x):=\argmin_{z\in C}{\Bre(z,x}),
\end{equation}
is well defined. It is called the \emph{Bregman projector} onto $C$ with respect to $\Leg$. The point $\Pc_{C}(x)$ 
is the \emph{Bregman projection} of $x$ onto $C$ with respect to $\Leg$ and is characterized by the following 
variational inequality \cite[Proposition~3.16]{BAU1997}
\begin{equation}\label{projC_ineq}
(\forall\, z \in C)\qquad
\big\langle z- \Pc_{C}(x), \nabla\Leg(x)-\nabla\Leg(\Pc_{C}(x)) \big\rangle\leq 0, 
\end{equation}
or equivalently, using \eqref{3p}, by the condition
\begin{equation}\label{projC}
(\forall\, z \in C
)\qquad
\Bre(z,x) \geq \Bre(z,\Pc_C(x)) + \Bre(\Pc_C(x),x).
\end{equation}
In the special case that $C$ is an affine set, in (\ref{projC}) and (\ref{projC_ineq}) equalities hold. We also define 
the \emph{Bregman distance to $C$} with respect to $\phi$ as
\begin{equation}
\label{eq:Bredist}
\Bdist_C\colon \inte(\dom \Leg) \to \left[0,+\infty\right[ \colon x \mapsto \inf_{z \in C} \Bre(z,x).
\end{equation}
As regards the Bregman projection, the following holds \cite{BAU2009}.

\begin{fact}
\label{prop:projproperties}
Let $C\subset \Hh$ be a nonempty closed convex set and let $\Pc_C$ be the Bregman projection onto $C$ as 
defined in \eqref{eq:Breproj}. Then the following hold.
\begin{enumerate}[label={\rm (\roman*)}]
\item\label{prop:projproperties_0} 
$(\forall\, x\in \inte(\dom\Leg))$\ \ $P_C(x) = x\ \Leftrightarrow\ x \in C\ \Leftrightarrow\ D_C(x)=0$.
\item\label{prop:projproperties_iii} 
$(\forall\, x \in C)\big(\forall\, y \in \inte(\dom\Leg)\big)$ $\Bre(x,P_C(y)) \leq \Bre(x,y) - \Bre(P_C(y), y)$.
\item\label{prop:projproperties_iv} 
Suppose that $\dom\Leg^*$ is open. Then the operator $P_C\colon\inte(\dom\Leg)\to\inte(\dom\Leg)\cap C$ and 
the function $D_C\colon\inte(\dom\Leg)\to\R$ are continuous.
\end{enumerate}
\end{fact}

We now consider Bregman projections onto affine sets. Hence, in the rest of the section we assume that 
\begin{equation}
C :=\{z\in \Hh \,\vert\, Az=b\},
\end{equation}
where $A\colon\Hh\to\Gg$ is a linear operator between Euclidean spaces and $b\in\Gg$. We wish to characterize 
the Bregman projector onto  $C$. Let $x\in\inte(\dom \Leg)$. Then, problem \eqref{eq:BregProj} turns into
\begin{equation}
\label{eq:mainprob}
\min_{\substack{z \in \Hh\\[0.3ex] A z = b}} \Bre(z,x).
\end{equation}
The dual problem of \eqref{eq:mainprob} (in the sense of Fenchel-Rockafellar) is 
\begin{equation}
\label{eq:maindualprob}
\min_{\lambda \in \Gg} 
\Leg^{*}(\nabla\Leg(x)+A^{*}\lambda)-\Leg^{*}(\nabla\Leg(x))-\scalarp{\lambda,b} =:
\Psi^x_{\Lc}(\lambda).
\end{equation}
We denote by $x_\star =  \Pc_{\Lc}(x)$ the unique solution of \eqref{eq:mainprob} and by $\lambda_\star$ a solution 
(not necessarily unique) of the dual problem \eqref{eq:maindualprob}. They are characterized by the following KKT
conditions
\begin{equation}
\label{eq:KKT}
x_\star \in \intdom,\quad
A x_\star = b, \quad\text{and}\quad \nabla \Leg(x)+ A^*\lambda_\star = \nabla\Leg(x_\star).
\end{equation}

A direct consequence of the KKT conditions are the following useful properties. 

\begin{fact}\label{prop:bregman_proj}
Let $x\in\inte(\dom \Leg)$ and
$C =\{z\in\Hh \,\vert\,  Az=b\}$. Let $x_\star = P_{C}(x)$ and let $\lambda_{\star}$ be a minimizer of $\Psi_{C}^{x}$. Then the following holds.

\begin{enumerate}[label={\rm (\roman*)}]
\item\label{prop:bregman_proj_0} 
$(\forall\ x,y\in\inte(\dom\Leg))$\ $P_C(x)=P_C(y) \Leftrightarrow \nabla\Leg(x)-\nabla\Leg(y)\in\range(A^*)$.
\item\label{prop:bregman_proj_ii} 
$x_{\star} = \nabla \Leg^{*}\left(\nabla\Leg(x) + A^{*}\lambda_{\star}\right)$ and 
$A \nabla \Leg^*(\nabla \Leg(x)+ A^*\lambda_\star)=b$.
\item\label{prop:bregman_proj_i} 
(Pythagora's theorem) $(\forall\, z\in \Lc\cap \dom\Leg)$\ $\Bre(z,x)  = \Bre(z,x_{\star}) +\Bre(x_{\star},x)$.
\end{enumerate}
\end{fact}

\begin{remark}
\label{rmk:hyperproj}
\normalfont
Suppose that $C$ is a hyperplane, that is, $C=\{x\in X\,\vert\, \scalarp{a,x}=b\}$. Then, by  
Fact~\ref{prop:bregman_proj}\ref{prop:bregman_proj_ii}, any dual solution $\lambda_\star$ satisfies
$\scalarp{a,\nabla \Leg^*(\nabla \Leg(x)+ \lambda_\star a)}=b$, which is an equation in $\R$ and hence can 
be easily solved via a number of iterative methods (e.g., bisection, gradient descent, Newton). Moreover, in this 
case the dual solution is unique since $\lambda_\star = \scalarp{a,\nabla\Leg(x_\star)-\nabla\Leg(x)} / \norm{a}^2$. 
In \cite{DT2007} several examples in which such equation can be solved explicitly are provided. 
\end{remark}

\begin{lemma}
\label{lm:aff_proj_1}
Let $C_1$ and $C_2$ be two closed affine sets such that $C_2 \subset C_1$ and let $x \in \inte(\dom \Leg)$.
Then, $P_{C_2}(x) = P_{C_2}(P_{C_1}(x))$ and $D_{C_2}(P_{C_1}(x))+ D_{C_1}(x) = D_{C_2}(x)$.
\end{lemma}
\begin{proof}
Let $x_i = P_{C_i}(x)$, $i=1,2$ and $z \in C_2$. Then using  Fact \ref{prop:bregman_proj} \ref{prop:bregman_proj_i}, 
$\Bre(x_2,x_1) + \Bre(x_1, x) = \Bre(x_2, x) \leq \Bre(z, x) = \Bre(z,x_1) + \Bre(x_1, x)$, which yields 
$\Bre(x_2,x_1) \leq \Bre(z,x_1)$. Hence $x_2 = P_{C_2}(x_1)$ and $D_{C_2}(x_1) + D_{C_1}(x) = D_{C_2}(x)$.
\end{proof}

\begin{lemma}\label{lm:main1}
Let $(x,\lambda)\in\inte(\dom \Leg) \times Y$ be such that $\nabla \Leg(x)+ A^*\lambda \in \inte(\dom \Leg^*)$. 
Then, the following hold.
\begin{enumerate}[label={\rm (\roman*)}]
\item\label{eq:20200406b}  
$(\forall\, z \in C\cap\dom\Leg)\quad\Psi^x_{\Lc}(\lambda)
= \Bre(z, \nabla \Leg^*(\nabla \Leg(x) + A^* \lambda)) - \Bre(z,x)$;
\item\label{main_rate11}
$(\forall\, z \in C\cap\dom\Leg)\quad
\Bre\big(z,\Pc_{C}(x)\big) \leq \Bre(z,\nabla\Leg^*(\nabla\Leg(x)+A^*\lambda))$.
\end{enumerate}
\end{lemma}

\begin{proof}
Since $A z = b$, it follows from \eqref{eq:maindualprob} and
Fact~\ref{prop:bregman_div}\ref{prop:bregman_div_vi} that 
\begin{align}
\nonumber \Psi^x_{\Lc}(\lambda) 
\nonumber &= \Leg^*(\nabla \Leg(x) + A^*\lambda) - \Leg^*(\nabla \Leg(x)) - \langle z, A^*\lambda \rangle \\
\nonumber&= D_{\Leg^*}(\nabla \Leg(x)+ A^* \lambda, \nabla \Leg(x)) + \langle x- z, A^*\lambda \rangle\\
\label{eq:20200406a}&= \Bre(x, \nabla \Leg^*(\nabla \Leg(x)+ A^* \lambda)) + \langle x- z, A^*\lambda \rangle.
\end{align}
Moreover, it follows from Fact~\ref{prop:bregman_div}\ref{prop:bregman_div_iv} that
\begin{equation*}
\Bre(z, \nabla \Leg^*(\nabla \Leg(x) + A^* \lambda)) =
\Bre(z,x) + \Bre(x,\nabla \Leg^*(\nabla \Leg(x) + A^* \lambda))
+ \langle x - z, A^* \lambda \rangle,
\end{equation*}
which together with \eqref{eq:20200406a} yields \ref{eq:20200406b}.	 Next, since $P_C(x) \in C$, weak 
duality yields $\Bre(P_C(x),x) \geq - \Psi^x_{\Lc}(\lambda)$. Then, by \ref{eq:20200406b}, 
$\Bre(P_C(x),x) \geq \Bre(z,x) - \Bre(z, \nabla \Leg^*(\nabla \Leg(x) + A^* \lambda))$. Statement 
\ref{main_rate11} follows by Pythagora's theorem given in 
Proposition~\ref{prop:bregman_proj}\ref{prop:bregman_proj_i}.
\end{proof}

\subsection{D-Fej\'er monotone sequences \cite{BAU2003b,CEN1997}}

Let $C\subset\Hh$ be a nonempty closed convex set. Let $\Leg$ be a Legendre function such that 
$C \cap \inte(\dom \Leg) \neq \varnothing$. A sequence $(x_k)_{k \in \N}$ in $\inte(\dom \Leg)$ is 
\emph{Bregman monotone} or \emph{$D$-Fej\'er monotone} w.r.t.~$C$ if
\begin{equation}
\label{eq:DFejer}
(\forall\, x \in C)(\forall\,k \in \N)\qquad \Bre(x, x_{k+1}) \leq \Bre(x,x_k).
\end{equation}

For $D$-Fej\'er monotone sequences, the following properties are known \cite[Proposition~4.1, Example~4.7, 
and Theorem~4.1(i)]{BAU2003b}. 

\begin{proposition}
\label{prop:DFejer}
Let $(x_k)_{k \in \N}$ be a D-Fejer monotone sequence with respect to $C$. Then
the following hold.
\begin{enumerate}[label={\rm (\roman*)}]
\item\label{prop:DFejer_i} 
$(\forall\, x \in C\cap \dom \Leg)\quad (\Bre(x,x_k))_{k \in \N}$ is decreasing.
\item\label{prop:DFejer_ib} 
$(D_C(x_k))_{k \in \N}$ is decreasing.
\item\label{prop:DFejer_ic} 
$(\forall\, k \in \N)(\forall\, p \in \N)\quad D_C(x_{k+p}) \leq D_C(x_k) - \Bre(P_C (x_k),P_C(x_{k+p}))$.
\item\label{prop:DFejer_ii} 
$(\forall\, x \in C\cap \dom \Leg)(\forall\, x^\prime \in C\cap \dom \Leg) \quad \scalarp{x - x^\prime, \nabla \Leg(x_k)}$ 
is convergent.
\item 
Suppose that $\dom \Leg^*$ is open. Then $(x_k)_{k \in \N}$ is bounded.
\item\label{prop:DFejer_iii} 
If all cluster points of $(x_k)_{k \in \N}$ lie in $C$, then $(x_k)_{k \in\N}$ converges to some point in 
$C \cap \inte(\dom \Leg)$.
\end{enumerate}
\end{proposition}

Concerning Proposition~\ref{prop:DFejer}\ref{prop:DFejer_iii}, we now give a result ensuring that the cluster points of 
$(x_k)_{k \in \N}$ lie in $C$. In the sequel we will consider the following \emph{sequential consistency assumption}
\cite{BAU2003b}. 
\begin{enumerate}[label={\rm \textbf{SC}},leftmargin=4.7ex]
\item\label{eq:SC} 
For all bounded sequences $(z_k)_{k \in \N}$ and $(y_k)_{k \in \N}$ in $\inte(\dom \Leg)$
\begin{equation*}
\Bre(z_k,y_k) \to 0\ \Rightarrow\ z_k - y_k \to 0.
\end{equation*}
\end{enumerate}

\begin{proposition}
\label{prop:DFejer2}
Suppose that $D_C(x_k) \to 0$ and that \ref{eq:SC} holds. Then $(x_k)_{k \in \N}$ 
converges to some point in $C \cap \inte(\dom \Leg)$.
\end{proposition}
\begin{proof}
Let $x \in C\cap \inte(\dom \Leg)$. It follows from Proposition~\ref{prop:DFejer}\ref{prop:DFejer_i} and
Fact~\ref{prop:bregman_div}\ref{prop:bregman_div_ix} that  $(x_k)_{k \in \N}$ is contained in the compact 
set $\{\Bre(x,\cdot ) \leq \Bre(x, x_0)\} \subset \inte(\dom \Leg)$. Hence the set of cluster points of 
$(x_k)_{k \in \N}$ is nonempty and contained in $\inte(\dom \Leg)$. Moreover, it follows from
Proposition~\ref{prop:DFejer}\ref{prop:DFejer_ic} (with $k=0$) and 
Fact~\ref{prop:bregman_div}\ref{prop:bregman_div_ix} that $(P_C(x_p))_{p \in \N}$ is contained in 
the compact set $\{ \Bre(P_C(x_0),\cdot) \leq D_C(x_0)\}$ and hence it is bounded. Let $x$ be a cluster point of 
$(x_k)_{k \in \N}$ and let $(x_{n_k})_{k \in \N}$ be a subsequence such that $x_{n_k} \to x$. Then, we saw that 
$x \in \inte(\dom \Leg)$. Moreover, $\Bre(P_C(x_{n_k}), x_{k_n}) = D_C(x_{n_k})\to 0$ and hence in virtue of
\ref{eq:SC}, we have that $P_C(x_{n_k}) - x_{n_k} \to 0$. Therefore, $P_C(x_{n_k}) \to x$,  which implies that 
$x \in C$, since $C$ is closed. Thus, we proved that all cluster points of $(x_k)_{k \in \N}$ lie in 
$C \cap \inte(\dom \Leg)$ and therefore, by Proposition~\ref{prop:DFejer}\ref{prop:DFejer_iii} we derive that 
$(x_k)_{k \in \N}$ converges to some point in $C$.
\end{proof}

\section{The method of Bregman projections}
\label{sec:BPM}

In this section we study the main properties of the method of Bregman projections in general convex 
feasibility problems. We first address the well known deterministic case for Problem~\ref{P1} in which 
the various projections are performed in a greedy manner. Then, we introduce the stochastic version of the 
algorithm which is designed for Problem~\ref{P2}. In either case we make the following basic assumption.

\begin{enumerate}[label={\rm \textbf{H0}},leftmargin=4.6ex]
\item\label{eq:stand1}
$C \neq X$, $\Leg$ is a Legendre function, $\dom \Leg^*$ is open, and $\inte(\dom \Leg) \cap C \neq \varnothing$.
\end{enumerate}
Moreover, in this section we will also consider the condition \ref{eq:SC} above.

We note that, referring to Problem~\ref{P1}, when $I$ is finite, say $I = \{1,\dots, n\}$, a standard implementation 
of Algorithm~\ref{algoABP} is that of cyclic projections, in which we have $\xi_k = (k\!\!\mod\! n)+1$ \cite{Alb1997}. 
In this work depending on the problem at hand we consider instead the following set control schemes.

\begin{enumerate}[label={\rm \textbf{C1}},leftmargin=4.6ex]
\item\label{eq:C1}
For every $k \in \N$, $\xi_k \in \argmax_{i \in I} \Bre(P_{C_i}(x_{k}), x_{k})$.
\end{enumerate}

\begin{enumerate}[label={\rm \textbf{C2}},leftmargin=4.6ex]
\item\label{eq:C2}
The $\xi_k$'s are $I$-valued random variables which are independent copies of $\xi$.
\end{enumerate}
\noindent
We call them \emph{greedy} and \emph{random}, respectively. The first one was considered in the pioneering work 
by Bregman \cite{BRE1967} and was called \emph{the most remote set control} in \cite{COM1993}. Instead, 
random set controls appear in \cite{NED2010} in the context of the orthogonal projection method and in \cite{HER2019}
in the study of stochastic fixed point equations. In Section~\ref{sec:adaptive} we will consider another type of random 
set control scheme which we call \emph{adaptive random} which is inspired by the work \cite{GMMN2020}.

\begin{remark}
\label{rmk:greedy-alternate}
According to Algorithm~\ref{algoABP}, for every $k \geq 1$, $x_k \in C_{\xi_{k-1}}$ and hence 
$D_{C_{\xi_{k-1}}}(x_k)=0$. So, if the greedy scheme \ref{eq:C1} is adopted, then 
$D_{C_{\xi_k}}(x_k)=\max_{i \in I} D_{C_i}(x_k)>0$ (otherwise $x_k \in \bigcap_{i\in I} C_i = C$ and 
the algorithm would stop) and hence $\xi_k \neq \xi_{k-1}$. This shows that if $I=\{1,2\}$, \ref{eq:C1} 
reduces to alternating projections.
\end{remark}

\subsection{Convex feasibility problem}
We start by approaching Problem~\ref{P1} via Algorithm~\ref{algoABP} with the greedy set control scheme \ref{eq:C1}.

\begin{proposition}
\label{prop:ABP}
Referring to Problem~\ref{P1}, suppose that \ref{eq:stand1} holds and let $(x_k)_{k \in \N}$ be generated 
by Algorithm~\ref{algoABP}. Then, the following hold.

\begin{enumerate}[label={\rm (\roman*)}]
\item\label{prop:ABP_i} 
$(\forall\, k \in \N) (\forall\, x \in C)\ \Bre(x,x_{k+1}) \leq \Bre(x,x_{k}) - \Bre(x_{k+1},x_{k})$.
Hence $(x_k)_{k \in \N}$ is D-Fejer monotone with respect to $C$.
\item\label{prop:ABP_ii} 
$\displaystyle\sum_{k=0}^{+\infty} \Bre(x_{k+1},x_{k}) <+\infty$\\[-1ex]
\end{enumerate}
\end{proposition}
 
\begin{proof}
\ref{prop:ABP_i}:
Let $k \in \N$. Since $x_{k+1} = P_{C_{\xi_k}}(x_{k})$, it follows from \eqref{projC} that
\begin{equation*}
(\forall\, x \in C_{\xi_k})\qquad
\Bre(x,x_{k}) \geq \Bre(x, x_{k+1}) + \Bre(x_{k+1},x_{k}).
\end{equation*}
Since $C \subset C_{\xi_k}$, the statement follows.

\ref{prop:ABP_ii}:
Let $x \in C \cap \dom \Leg$. Then, by \ref{prop:ABP_i} we derive that 
$\Bre(x_{k+1},x_{k}) \leq \Bre(x,x_{k}) - \Bre(x,x_{k+1})$ and hence 
$\sum_{k=0}^{+\infty}\Bre(x_{k+1},x_{k}) \leq \Bre(x,x_{0})<+\infty$.
\end{proof}

The following is essentially Theorem~2 in \cite{BRE1967}.

\begin{theorem}[Greedy control scheme]
\label{thm:ABP2}
Referring to Problem~\ref{P1}, suppose that \ref{eq:stand1} and \ref{eq:SC} hold, and that $(x_k)_{k \in \N}$ is 
generated by Algorithm~\ref{algoABP} with the greedy set control scheme \ref{eq:C1}. Then $D_C(x_k)\!\to\!0$ and 
$x_k \to \hat{x} \in C\cap \inte(\dom\Leg)$.
\end{theorem}

\subsection{Stochastic convex feasibility problems}

We now consider Problem~\ref{P2}. We denote by $\mu$ the distribution of the random variable $\xi$ and by 
$(\Omega, \mathfrak{A}, \PP)$ the underlying probability space. We recall that the set-valued mapping 
$i\in I \to C_i \subset \Hh$ (with closed values) is measurable if for all Borel set $Z \subset \Hh$, 
$\{ i \in I \,\vert\, C_i\cap Z \neq \varnothing\} \in \mathcal{I}$. Therefore for all $x \in \Hh$, $\{i \in I \,\vert\, x \in C_i\}$ 
is measurable and thanks to \cite[Lemma III.39]{CAS1977} (with $\varphi(i,z) = - \Bre(z,x)$ and $\Sigma(i) = C_i$) 
we have that, for every $x \in \inte(\dom \Leg)$, the function $i \in I \to D_{C_i}(x) \in \R$ is also measurable.
We will study Algorithm~\ref{algoABP} adopting the random set control scheme \ref{eq:C2}, meaning that the 
$\xi_k$'s are independent copies of the random variable $\xi$. Note that now $x_k$ is a random variable and 
more precisely $x_k = x_k (\xi_0, \dots, \xi_{k-1})$, so that $x_k$ and $\xi_k$ are independent random variables.
We denote by $\mathfrak{X}_k$ the $\sigma$-algebra generated by $x_0, \dots x_{k}$. Finally, we set
\begin{align}
\label{eq:20200611b}
\overline{D}_{\hatC}\colon \inte(\dom \Leg) &\to \R\colon x \mapsto \EE[D_{C_\xi}(x)].
\end{align}

\begin{remark}
\normalfont
Since $\{ i \in I \,\vert\, x \in C_i \} \in \mathcal{I}$ and $\xi$ is mesurable, we have 
$\{ x \in C_\xi\} = \xi^{-1}(\{ i \in I \,\vert\, x \in C_i \}) \in \mathfrak{A}$. Moreover, for every $x \in \inte(\dom\Leg)$, 
since $i \in I \to D_{C_i}(x)$ is measurable, $D_{C_\xi}(x)$ is measurable too. Note that
\begin{align}
\label{eq:20200809a}
\PP(\{x \in C_\xi\}) &= \PP(\xi^{-1}(\{i \in I \,\vert\, x \in C_i\})) = \mu(\{i \in I \,\vert\, x \in C_i\})\\
\EE[D_{C_\xi}(x)] &= \int_{\Omega} D_{C_{\xi(\omega)}}(x) \PP(d \omega) 
= \int_{\mathcal{I}} D_{C_i}(x) \mu(d\,\!i).
\end{align}
Therefore, the definitions of $C$, in Problem~\ref{P2}, and \eqref{eq:20200611b} depend only on the 
distribution $\mu$ of $\xi$. Finally note that by \eqref{eq:20200809a} one derives that
\begin{equation}
\label{eq:20200810a}
\PP(\{x \in C_\xi\}) = 1\ \Leftrightarrow\ \mu(\{i \in I \,\vert\, x \in C_i\})=1.
\end{equation}
\end{remark}
\vspace{-1ex}
\begin{proposition}
\label{prop:stochasticintersection2}
Referring to Problem~\ref{P2} and assuming \ref{eq:stand1}, 
the following hold.
\begin{enumerate}[label={\rm (\roman*)}]
\item\label{prop:stochasticintersection2_0} 
The function $\overline{D}_{\hatC}\colon \inte(\dom \Leg) \to \R$ is lower semicontinuous.
\item\label{prop:stochasticintersection2_i} 
$(\forall\, x \in \inte(\dom \Leg))\ x \in \hatC\ \Leftrightarrow\ \overline{D}_{\hatC}(x)=0$.
\item\label{prop:stochasticintersection2_ii} 
There exists a $\mu$-negligible set $J\subset I$ such that, $\hatC = \bigcap_{j \in I\setminus J} C_{j}$. 
Moreover, if $(\xi_k)_{k \in \N}$ is a sequence of $I$-valued random variables each one having the same 
distribution of $\xi$, then there exists a $\PP$-negligible set $N$ such that, for every $k \in \N$, 
$\hatC = \bigcap_{\omega \in \Omega\setminus N} C_{\xi_k(\omega)}$.
\end{enumerate}
\end{proposition}
 
\begin{proof}
\ref{prop:stochasticintersection2_0}:
Let $(x_k)_{k \in \N}$ be a sequence in $\inte(\dom\Leg)$ and $x \in \inte(\dom \Leg)$ be such that $x_k \to x$ 
and let $z \in C\cap\inte(\dom\Leg)$. Then, there exists a $\PP$-negligible set $N\subset \Omega$ such that, 
for every $\omega \in \Omega\setminus N$, $z \in C_{\xi(\omega)}$. It follows from 
Fact~\ref{prop:projproperties}\ref{prop:projproperties_iv} that, for every $\omega \in \Omega\setminus N$,
$D_{C_{\xi(\omega)}}\colon \inte(\dom\Leg)\to \R$ is continuous and hence 
$D_{C_{\xi(\omega)}}(x_k)\to D_{C_{\xi(\omega)}}(x)$. Therefore, $D_{C_{\xi}}(x_k) \to D_{C_{\xi}}(x)$ almost 
surely and hence by Fatou's lemma we have $\EE[D_{C_{\xi}}(x)] \leq \liminf_{k\to+\infty} \EE[D_{C_{\xi}}(x_k)]$.

\ref{prop:stochasticintersection2_i}:
Let $x \in \inte(\dom\Leg)$. Since the integrand in $\overline{D}_{\hatC}(x)$ is positive, it follows from 
Fact~\ref{prop:projproperties}\ref{prop:projproperties_0} that 
$\overline{D}_{\hatC}(x)=0 \ \Leftrightarrow\ D_{C_\xi}(x)=0\ \PP\text{-a.s.}$ $\Leftrightarrow\ x \in C_\xi\ \PP\text{-a.s.}$

\ref{prop:stochasticintersection2_ii}: Let $Q$ be a countable dense subset of $\hatC$ and let $x \in Q$. Then, 
since $x \in \hatC$, it follows from \eqref{eq:20200810a} that there exists a $\mu$-negligible subset 
$J_{x} \subset I$ such that $x \in C_{i}$ for every $i \in I\!\setminus\! J_{x}$. Set $J = \bigcup_{x \in Q} J_x$.
Then $J$ is $\mu$-negligible and, for every $i \in I\!\setminus\! J$, $x\in C_{i}$; hence, 
$x \in \bigcap_{i \in I \setminus J} C_{i}$. We then proved that $Q \subset \bigcap_{i \in I \setminus J} C_{i}$.
Since this latter intersection is closed we have $\hatC = \cl(Q) \subset \bigcap_{i \in I \setminus J} C_{i}$. 
On the other hand, if $x \in \bigcap_{i \in I \setminus J} C_{i}$, then, we have $x \in \hatC$, again by 
\eqref{eq:20200810a}. Suppose now that $(\xi_k)_{k \in \N}$ is a sequence of $I$-valued random variables 
each one distributed according to $\mu$. Set $N = \bigcup_{k \in \N} \xi_k^{-1}(J)$ and let $k \in \N$. 
We first prove that $C \subset \bigcap_{\omega \in \Omega \setminus N} C_{\xi_k(\omega)}$. Let $x \in C$. 
Then, since $\PP(\xi_k^{-1}(J))=\mu(J)=0$, $N$ is $\PP$-negligible and for every $\omega \in \Omega\setminus N$ 
we have $\xi_k(\omega) \in I\setminus J$ and hence $x \in C_{\xi_k(\omega)}$. 
Thus, $x \in \bigcap_{\omega \in \Omega \setminus N} C_{\xi_k(\omega)}$. The other inclusion follows from 
the true definition of $C$ in Problem~\ref{P2}, noting that if 
$x \in \bigcap_{\omega \in \Omega \setminus N} C_{\xi_k(\omega)}$, then 
$\PP(\{x \in C_\xi\}) = \PP(\{x \in C_{\xi_k}\})=1$.
\end{proof}

\begin{remark}
\normalfont
Proposition~\ref{prop:stochasticintersection2}\ref{prop:stochasticintersection2_ii}
is essentially  \cite[Lemma~2.4 and Corollary~2.6]{HER2019}.
\end{remark}

\begin{theorem}[Random control scheme]
\label{thm:conv2}
Referring to Problem~\ref{P2}, suppose that assumption \ref{eq:stand1} holds. Let the function 
$\overline{D}_C\colon \inte(\dom \Leg) \to \R$ be defined as in \eqref{eq:20200611b}. If  $(x_k)_{k \in \N}$ is 
generated by Algorithm~\ref{algoABP} with the random set control scheme \ref{eq:C2}, then the following hold.
\begin{enumerate}[label={\rm (\roman*)}]
\item\label{thm:conv2_iii} 
The sequence $(x_k)_{k \in \N}$ is  $D$-Fej\'er monotone w.r.t.~$C$ $\PP$-a.s.~and contained in a compact 
subset of $X$ $\PP$-a.s. 
\item\label{thm:conv2_i} 
$(D_{\hatC}(x_k))_{k \in \N}$ is $\PP$-a.s.~decreasing and $\overline{D}_{\hatC}(x_k) \to 0$ $\PP$-a.s.
\item\label{thm:conv2_iv} 
There exists an $X$-valued random variable $\hat{x}$ and a subsequence $(x_{n_k})_{k \in \N}$ such that
$\hat{x} \in \hatC\cap \inte(\dom\Leg)$ $\PP$-a.s.~and $x_{n_k} \to \hat{x}$ in distribution.
\item\label{thm:conv2_ii} 
$D_{\hatC}(x_k) \to 0$ $\PP$-a.s.~and $\EE[D_{\hatC}(x_k)] \to 0$.
\item\label{thm:conv2_v} 
If \ref{eq:SC} holds, then there exists an $X$-valued random variable $\hat x$ such that 
$\hat x\in \inte(\dom\Leg)\cap C$ $\PP$-a.s. and $x_k \to \hat x$ $\PP$-a.s.~and $\EE[\norm{x_k- \hat{x}}^2] \to 0$.
\end{enumerate}
\end{theorem}
 
\begin{proof}
\ref{thm:conv2_iii}:
Proposition \ref{prop:stochasticintersection2}\ref{prop:stochasticintersection2_ii}  yields that there exists a 
$\PP$-negligible set $N$ such that $\hatC = \bigcap_{\omega \in \Omega\setminus N} C_{\xi_k(\omega)}$ 
for every $k\in\N$. Let $\omega \in \Omega\setminus N$ and let $x \in \hatC$. Then, for every $k\in\N$,  
$x \in C_{\xi_k(\omega)}$, and, hence, since $x_{k+1}(\omega) = \Pc_{C_{\xi_k(\omega)}}(x_{k}(\omega))$, 
it follows from \eqref{projC} that, for every $k \in \N$,
\begin{equation}
\label{eq:20200526a}
\Bre(x,x_{k+1}(\omega)) \leq \Bre(x,x_{k}(\omega)) - \Bre(x_{k+1}(\omega),x_{k}(\omega)).
\end{equation}
This shows that $(x_k)_{k\in\N}$ is $D$-Fej\'er monotone w.r.t $C$ $\PP$-a.s.~and hence, 
Proposition~\ref{prop:DFejer} yields that $(\Bre(x,x_k))_{k \in \N}$ is $\PP$-a.s.~decreasing. So, if we pick 
$x \in C \cap \inte(\dom\Leg)$, it follows from Fact~\ref{prop:bregman_div}\ref{prop:bregman_div_ix} that
$(x_k)_{k \in \N}$ is $\PP$-a.s.~contained in the compact sublevel set 
$\{\Bre(x, \cdot) \leq \Bre(x, x_0)\} \subset \inte(\dom\Leg)$.

\ref{thm:conv2_i}-\ref{thm:conv2_iv}: 
It follows from \ref{thm:conv2_iii} and Proposition \ref{prop:DFejer}\ref{prop:DFejer_ib} that $(D_C(x_k))_{k\in\N}$ is 
$\PP$-a.s.~decreasing. Also, \ref{thm:conv2_iii} yields that the sequence $(x_k)_{k \in \N}$ is uniformly essentially
bounded, i.e., $\sup_{k \in \N} \esssup \norm{x_k}<+\infty$. So, Prokhorov's theorem ensures that
there exists a subsequence $(x_{n_k})_{k \in \N}$ converging in distribution to some random vector $\hat{x}$.
Let $x \in C\cap \inte(\dom\Leg)$. Since, in virtue of Fact~\ref{prop:bregman_div}\ref{prop:bregman_div_ix}, 
$\Bre(x,\cdot)$ is positive and lower semicontinuous on $X$, Portmanteau theorem 
(see e.g., \cite[Theorem~2.8.1]{ASH2000}) yields that
\begin{equation}
\EE[\Bre(x,\hat{x})]\leq \liminf_{k\to \infty} \EE[\Bre(x,x_{n_k})] \leq \Bre(x,x_0)<+\infty.
\end{equation}
Thus, $\Bre(x,\hat{x})<+\infty$ $\PP$-a.s.~and hence $\hat{x} \in \inte(\dom\Leg)$ $\PP$-a.s. Now, taking the 
conditional expectation in \eqref{eq:20200526a} and using Fact~\ref{f:20190112d}, we get
\begin{align*}
\EE[ \Bre(x,x_{k+1}) \,\vert\, \mathfrak{X}_{k}] &\leq \Bre(x,x_{k}) 
- \EE[\Bre(P_{C_{\xi_k}} (x_{k}),x_{k}) \,\vert\, \mathfrak{X}_{k}]\\
&= \Bre(x,x_{k}) - \EE[D_{C_{\xi_k}}(x_{k}) \,\vert\, \mathfrak{X}_{k}]\\
& = \Bre(x,x_{k})  - \overline{D}_{\hatC}(x_{k})
\end{align*}
and hence $\EE[\Bre(x,x_{k+1})] \leq \EE[\Bre(x,x_{k})] - \EE[\overline{D}_{\hatC}(x_{k})]$.
This shows that  
 $(\EE[\Bre(x,x_k)])_{k \in \N}$ is decreasing, hence convergent, and also that
$\EE[\sum_{k=0}^{+\infty} \overline{D}_{\hatC}(x_{k})] = \sum_{k=0}^{+\infty} \EE[\overline{D}_{\hatC}(x_{k})]<+\infty$.
Then we have $\EE[\overline{D}_{\hatC}(x_{k})] \to 0$ and $\sum_{k=0}^{+\infty} \overline{D}_{\hatC}(x_{k})<+\infty$
$\PP$-a.s., so that $\overline{D}_{\hatC}(x_{k})\to 0$ $\PP$-a.s. Now, since, in virtue of
Proposition~\ref{prop:stochasticintersection2}\ref{prop:stochasticintersection2_0}, the function 
$\overline{D}_{\hatC}\colon \inte(\dom\Leg) \to \R$ is lower semicontinuous and $x_{n_k} \to \hat{x}$ in distribution,
another application of Portmanteau theorem gives 
$\EE[\overline{D}_{\hatC}(\hat{x})] \leq \liminf_{k\to+\infty} \EE[\overline{D}_{\hatC}(x_{n_k})] =0$.
This, yields that $\overline{D}_{\hatC}(\hat{x}) =0$ $\PP$-a.s.~and hence, by Proposition~\ref{prop:stochasticintersection2}\ref{prop:stochasticintersection2_i}, that $\hat{x} \in \hatC$ \PP-a.s.

\ref{thm:conv2_ii}: 
From \ref{thm:conv2_i} we have that $(D_{\hatC}(x_k))_{k \in \N}$ converges almost surely to some nonnegative 
and finite random variable $\zeta$ and that $D_{\hatC}(x_k) \leq D_{\hatC}(x_0)$ $\PP$-a.s. Thus, it follows from 
the Lebesgue's dominated convergence theorem that $\EE[D_{\hatC}(x_k)] \to \EE[\zeta]$. Now, from 
\ref{thm:conv2_iv} we have that $x_{n_k} \to \hat{x}$ in distribution and $\hat x\in C \cap \inte(\dom\Leg)$ $\PP$-a.s. 
However, $D_{\hatC}$ is continuous on $\inte(\dom \Leg)$, and hence, in virtue of Fact~\ref{fact_convRV},
$D_{\hatC}(x_{n_k}) \to D_{\hatC}(\hat{x})=0$ in distribution. Therefore, $\EE[D_{\hatC}(x_{n_k})] \to 0$. This shows 
that $\EE[\zeta]=0$ and hence $\zeta = 0$ $\PP$-a.s. So, we have that $D_{\hatC}(x_k) \to 0$ $\PP$-a.s.

\ref{thm:conv2_v}: 
By item \ref{thm:conv2_iii}, there exists a negligible set $N \subset \Omega$ such that, for all 
$\omega \in \Omega\setminus N$, $(x_k(\omega))_{k \in \N}$ is $D$-Fej\'er monotone w.r.t.~$C$. Then 
Proposition~\ref{prop:DFejer2} yields that, for every $\omega \in \Omega\setminus N$, $(x_k(\omega))_{k \in \N}$
converges to some point in $C\cap \inte(\dom \Leg)$. So, convergence $\PP$-a.s.~follows, which, by
Fact~\ref{fact_convRV}, implies convergence in mean square.
\end{proof}

\begin{remark}
\label{rmk:repcontrol}
\normalfont
In \cite{CEN1996}, in relation to Algorithm~\ref{algoABP} and for finite index set $I$, a repetitive control sequence 
$(\xi_k)_{k \in \N}$ was used, meaning that for every $i \in I$, the set $\{k \in \N \,\vert\, \xi_k = i\}$ is infinite. This type 
of control was also called random in \cite{BAU1997}. We show here that this concept can indeed cover the stochastic
setting analyzed in Theorem~\ref{thm:conv2}, when $I$ is finite. Indeed, assume that, for every $i \in I$, 
$\PP(\xi = i) = p_i>0$. Let $i \in I$ and set, for every $k \in \N$, $S_k = \{\xi_k = i\}$. Then 
$\sum_{k\in \N} \PP(S_k) = +\infty$ and, since $(\xi_k)_{k \in \N}$ is an independent sequence of random variables, 
we have that $(S_k)_{k \in \N}$ is an independent sequence of events. Hence, by the second Borel-Cantelli 
lemma \cite[Theorem~2.3.6]{DUR2010}, we derive that $\PP(\limsup_k S_k)=1$. Note that 
$\omega \in \limsup_k S_k\ \Leftrightarrow\ \{k \in \N\,\vert\, \omega \in S_k\}$ is infinite. 
Moreover, since $\tilde{\Omega}_i = \limsup_k S_k$ is a set of probability one, so is 
$\tilde{\Omega} = \cap_{i \in I} \tilde{\Omega}_i$. Therefore, if we pick $\omega \in \tilde{\Omega}$, we have that 
$(\xi_k(\omega))_{k \in \N}$ is a repetitive control sequence and hence the almost sure convergence
in Theorem~\ref{thm:conv2}\ref{thm:conv2_ii}-\ref{thm:conv2_v} can be derived from 
\cite[Theorem~8.1]{BAU1997} (see also \cite[Theorem~3.2]{CEN1996}).
\end{remark}

\section{Convergence of the method of Bregman projections for affine sets}
\label{sec:BPMaff}

In this section we analyze the convergence of the Bregman projection method for affine feasibility problems. 
We prove global and local Q-linear convergence of the method, providing also explicit global and local rates. We cover 
three options for the set control sequence: greedy, random, and adaptive random. 

We will make the following assumption.

\begin{enumerate}[label={\rm \textbf{H1}},leftmargin=4.6ex]
\item\label{eq:AS3}
The sets $C_i$'s are affine, i.e., for every $i\in I$, $C_i = \{x \in \Hh \,\vert\, A_i x = b_i\}$ for some nonzero linear 
operator $A_i\colon \Hh \to \Gg_i$  and some $b_i \in \Gg_i$.
\end{enumerate}
In this situation in both Problem~\ref{P1} and Problem~\ref{P2} the set $C$ is affine. We therefore let 
$A\colon \Hh \to \Gg$ be a linear operator and $b \in \Gg$ such that 
\begin{equation}
C = \{x \in \Hh \,\vert\, A x = b\}.
\end{equation}
As before, we will make assumption~\ref{eq:stand1} and in addition, depending on the fact that we are dealing 
with Problem~\ref{P1} or Problem~\ref{P2}, we will also make one of the following technical assumptions.

\begin{enumerate}[label={\rm \textbf{H2\textsubscript{1}}},leftmargin=5.5ex]
\item\label{eq:AS3a} 
$\Leg^*$ is twice differentiable and, for all $x\!\in\!\inte(\dom\Leg)\cap C$, 
$A \nabla^2\Leg^*(\nabla\Leg(x)) \neq 0$ and
\begin{equation*}\label{eq:UBa}
\sup_{i \in I}\norm{A_i^*(A_i \nabla^2\Leg^*(\nabla\Leg(x)) A_i^*)^\dagger A_i}<+\infty.
\end{equation*} 
\end{enumerate}
\begin{enumerate}[label={\rm \textbf{H2\textsubscript{2}}},leftmargin=5.6ex]
\item\label{eq:AS3b}  
$\Leg^*$ is twice differentiable and, for all $x\!\in\! \inte(\dom\Leg)\cap C$, 
$A \nabla^2\Leg^*(\nabla\Leg(x)) \neq 0$ and
\begin{equation*}
\esssup\norm{A_\xi^*(A_\xi \nabla^2\Leg^*(\nabla\Leg(x)) A_\xi^*)^\dagger A_\xi}<+\infty.
\end{equation*} 	
\end{enumerate}
The result below shows that \ref{eq:AS3a} and \ref{eq:AS3b} hold in a number of significant cases.

\begin{proposition}\label{prop:AF3a}
Suppose that $\Leg^*$ is twice differentiable. Then, referring to Problem~\ref{P1} (resp. Problem~\ref{P2}),
assumption~\ref{eq:AS3a} (resp. assumption~\ref{eq:AS3b}) holds if one of the following conditions is met:
\begin{enumerate}[label={\rm (\alph*)}]
\item\label{prop:AF3a_i} for every $x\in\inte(\dom\Leg)\cap C$, $\nabla^2\Leg^*(\nabla\Leg(x))$ is invertible, 
\item\label{prop:AF3a_ii} $\Leg$ is twice differentiable on $\inte(\dom\Leg)$,
\item\label{prop:AF3a_iii} $I$ is finite and for $x\in \inte(\dom\Leg)\cap C$, $A \nabla^2\Leg^*(\nabla\Leg(x)) \neq 0$, 
\item\label{prop:AF3a_iv}  $I$ is finite, $X=\R^n$, $0 \notin C$, and $\Leg(x)=(1/p)\norm{x}_p^p$, for $p\in\left]1,2\right[$.
\end{enumerate}
\end{proposition}	

\begin{proof}
We will prove only the case \ref{eq:AS3a}. The other is similar.

\ref{prop:AF3a_i}: 	
Let $x\in \inte(\dom\Leg)\cap C$ and set $H=[\nabla^2\Leg^*(\nabla\Leg(x))]^{1/2}$. Since $H^2$ is invertible, 
we have $A \neq 0\ \Rightarrow\ A H^2 \neq 0$. Moreover, it follows from Fact~\ref{fact_inf} that 
$Q_i(x) = HA_i^*(A_i H^2 A_i^*)^\dagger A_i H$ is the orthogonal projector onto $\range(H A_i^*)$. Since $H$
is invertible, we have $A_i^*(A_i H^2 A_i^*)^\dagger A_i = H^{-1} Q_i(x) H^{-1}$ and hence
$\Vert A_i^*(A_i H^2 A_i^*)^\dagger A_i\rVert \leq \lVert H^{-1}\rVert \norm{Q_i(x)} \lVert H^{-1}\rVert = \lVert H^{-1}\rVert^2$. Thus, the last condition in \ref{eq:AS3a} follows.
	
\ref{prop:AF3a_ii}: 
Since both $\Leg$ and $\Leg^*$ are twice differentiable, their Hessians are inverse to each other. Hence, \ref{prop:AF3a_i}, and a fortiori \ref{eq:AS3a} holds. 
	
\ref{prop:AF3a_iii}: 
If $I$ is finite, then the last condition in \ref{eq:AS3a} trivially holds. 
	
\ref{prop:AF3a_iv}: 
First, note that $\Leg^*(y) = q^{-1}\norm{y}^q$, where $p^{-1}+q^{-1}=1$, and that $\dom \Leg =\dom \Leg^* =\R^n$. 
Next, let $x\in \inte(\dom \Leg) \cap C$ and  $y= \nabla\Leg(x)$. Then, for every $j\in \{1,\ldots, n\}$, 
$y_j = \mathrm{sgn}(x_j)\abs{x_j}^{p-1}$. Hence, $\nabla^2\Leg^*(y)$ is a diagonal matrix such that 
$[\nabla^2\Leg^*(y)]_{j,j} = (q-1)\abs{y_j}^{q-2} = (q-1)\abs{x_j}^{(p-1)(q-2)}$, for every $j\in \{1,\ldots, n\}$. 
Now, since $x\neq0$, if we define the vector $v\in\R^n$ as
\begin{equation}
(\forall\,j\in [n])\quad v_j = 
\begin{cases}
(q-1)^{-1}\mathrm{sgn}(x_j)\abs{x_j}^{1-(p-1)(q-2)}&\text{if } x_j\neq0\\
0&\text{otherwise},
\end{cases}	
\end{equation}
we have $v\neq0$ and $\nabla^2\Leg^*(y) v = x$, and hence $A\nabla^2\Leg^*(y) v = Ax = b\neq 0$ 
(since $0 \notin C$). So, the first part of \ref{eq:AS3a} follows. The last condition in  \ref{eq:AS3a} follows 
from \ref{prop:AF3a_iii}.	
\end{proof}

\subsection{Greedy set control scheme}

We  address Problem~\ref{P1} assuming \ref{eq:stand1}, \ref{eq:AS3}, and \ref{eq:AS3a}. We first give a 
general convergence theorem and then we study rate of convergence going through several technical 
results and finally providing a theorem.

\begin{theorem}
\label{thm3}
Under the assumptions of Problem~\ref{P1} and \ref{eq:stand1}, let $(x_k)_{k \in \N}$ be generated by 
Algorithm~\ref{algoABP} using any (deterministic) sequence $(\xi_k)_{k \in \N}$ in $I$ and let 
$x_\star = P_C(x_0)$. Then the following hold.

\begin{enumerate}[label={\rm (\roman*)}]
\item\label{thm3_0}
For every $k \in \N$, $\nabla \Leg(x_\star) - \nabla \Leg(x_k) \in \range(A^*)$.
\item\label{thm3_i} 
$(\forall\, k \in \N)\quad P_C(x_k)= x_\star\ $ and $\ D_C(x_{k+1}) + D_{C_{\xi_k}}(x_{k}) = D_C(x_{k})$.
\item\label{thm3_ii}  
Suppose that \ref{eq:SC} and \ref{eq:C1} hold. Then  $x_k \to x_\star$.
\end{enumerate}
\end{theorem}

\begin{proof}
\ref{thm3_0}: 
Let, for every $i \in I$, $C_i = \{x \in \Hh \,\vert\, A_i x = b_i\}$ for some linear operator $A_i\colon \Hh \to \Gg$ 
and $b_i \in \Gg$. Since $C \subset C_i$, we have $\Ker(A) \subset \Ker(A_i)$ and hence 
$\range(A_i^*) = \Ker(A_i)^\perp \subset \Ker(A)^\perp = \range(A^*)$. Moreover,  \eqref{eq:KKT} yields that 
$\nabla\Leg(x_{\star})=\nabla\Leg(x_0) + v_{\star}$ for some $v_\star \in \range(A^*)$ and, for every 
$k \in \N$, $\nabla\Leg(x_{k+1})=\nabla\Leg(x_{k}) + v_{\xi_{k}}$ for some
 $v_{\xi_k} \in \range(A_{\xi_k}^*) \subset \range(A^*)$.
Therefore, $\nabla\Leg(x_k)=\nabla\Leg(x_{0}) + \sum_{j=0}^{k-1} v_{\xi_j}$
and hence $\nabla\Leg(x_\star) - \nabla\Leg(x_k) =v_\star - \sum_{j=0}^{k-1} v_{\xi_j} \in \range(A^*)$.

\ref{thm3_i}:
Let $k \in \N$. Since $C \subset C_{\xi_k}$ and $x_{k+1} = P_{C_{\xi_k}}(x_{k})$, Lemma~\ref{lm:aff_proj_1} 
(with $x=x_{k}$) yields that
\begin{equation}
\label{eq:20201021b}
P_C(x_{k}) = P_C(x_{k+1})\quad\text{and}\quad
D_C(x_{k+1}) + D_{C_{\xi_k}}(x_{k}) = D_C(x_{k}).
\end{equation}
Thus, the second part of the statement follows immediately, while the first part follows by
applying the first equation in \eqref{eq:20201021b} recursively.

\ref{thm3_ii}: 
By Theorem~\ref{thm:ABP2} we have $x_k \to \hat{x}$ for some $\hat{x} \in C\cap \inte(\dom \Leg)$.
Since $P_C$ is continuous on $\inte(\dom \Leg)$, we have $P_C(x_k) \to P_C(\hat{x}) = \hat{x}$.
However, it follows from \ref{thm3_i} that $P_C(x_k) = x_\star$, for every $k \in \N$.
Therefore, $\hat{x} =  x_\star$. 
\end{proof}

Next, we provide Lemma \ref{lm:gamma_I} containing the definition of $\gamma_\C$, which governs the local 
linear rate. Then, we introduce Lemma \ref{lm:sigma_I} that defines $\sigma_\C$, which is critical for obtaining 
the global convergence rate. Finally, we conclude with Theorem \ref{thm:main2} covering the Q-linear convergence 
of the greedy Bregman projection method for affine sets.

\begin{lemma}\label{lm:gamma_I}
Referring to Problem~\ref{P1}, suppose that assumptions \ref{eq:stand1}, \ref{eq:AS3}, and \ref{eq:AS3a} hold. 
Let, for every $x\in\inte(\dom\Leg)\cap C$, $Q_i(x)$ be the orthogonal projection onto 
$\range([\nabla^2 \Leg^*(\nabla\Leg(x))]^{1/2} A_i^*)$, $V(x)=\range([\nabla^2 \Leg^*(\nabla\Leg(x))]^{1/2} A^*)$, and 
\begin{equation}
\label{eq:main2_i}
\gamma_{\C}(x) = \inf_{v \in V(x) \setminus \{0\}}
\sup_{i \in I} \frac{\norm{Q_i(x) v}^2}{ \norm{v}^2}.
\end{equation}
Then, for every $x\in\inte(\dom\Leg)\cap C$, $\gamma_{\C}(x)\in\left]0,1\right]$. Moreover, for every 
$\varepsilon\in\left]0,1\right[$ there exists $\delta>0$, such that for every $x\in\inte(\dom\Leg)$, 
\begin{equation}\label{eq:20200930a}
D_C(x)< \delta\ \implies\ \inf_{i\in I} D_C(P_{C_i}(x))
\leq \dfrac{1+\varepsilon}{1-\varepsilon} [1-\gamma_{\C}(P_C(x))]\,D_{C}(x).
\end{equation}
\end{lemma}
\vspace{-2ex}
\begin{proof}
Let $x \in \inte(\dom \Leg)$ and $y=\nabla\Leg(x)$. First of all we note that assumption \ref{eq:AS3a} yields 
$[\nabla^2 \Leg^*(y)]^{1/2}[\nabla^2 \Leg^*(y)]^{1/2} A^* = [\nabla^2 \Leg^*(y)] A^* \neq 0$ and hence 
$[\nabla^2 \Leg^*(y)]^{1/2} A^* \neq 0$, which in turns implies $V(x)\neq\{0\}$.
Clearly, $\gamma_{\C}(x) = \min_{v \in V(x), \norm{v}=1} \sup_{i\in I} \norm{Q_i(x) v}^2$ and hence 
$\gamma_{\C}(x) \in[0,1]$. Now assume, by contradiction, that  $\gamma_{\C}(x)=0$. Then, there exists 
$v \in V(x)$, with $\norm{v}=1$, such that
\begin{equation}
(\forall\, i \in I)\quad v\in \Ker(Q_i(x))= \range([\nabla^2 \Leg^*(y)]^{1/2} A_i^*)^\perp 
= \Ker( A_i [\nabla^2 \Leg^*(y)]^{1/2}),
\end{equation}
that is, $[\nabla^2 \Leg^*(y)]^{1/2} v \in\Ker(A_i)$. Then, since $\bigcap_{i\in I}\Ker(A_i)=\Ker(A)$, we have 
that $[\nabla^2 \Leg^*(y)]^{1/2} v \in\Ker(A)$, and, hence, $v\in \Ker( A [\nabla^2 \Leg^*(y)]^{1/2})=V(x)^{\perp}$. 
Thus, since $v\neq 0$, we obtain a contradiction, and hence necessarily $\gamma_{\C}(x)>0$.

The proof of \eqref{eq:20200930a} is quite technical. Therefore, we will proceed through 6 steps. 
First note that if $x\in\inte(\dom\Leg)\cap C$, then, recalling Fact~\ref{prop:projproperties}\ref{prop:projproperties_0}, 
we have $0=D_C(x)=D_C(P_{C_i}(x))$, for every $i \in I$, hence the inequality in \eqref{eq:20200930a} trivially holds.
Therefore, in the following we let $x\in\inte(\dom\Leg) \setminus C$ and set $x_\star=P_C(x)$, $y = \nabla\Leg(x)$ and 
$y_\star = \nabla \Leg(x_\star)$. Additionally, for the sake of brevity, we set $H = [\nabla^2 \Leg^*(y_\star)]^{1/2}$.
Let $i \in I$ and set $x_i = P_{C_i}(x)$ and $y_i = \nabla \Leg(x_i)$.  

\emph{Step }$1$: We have 
\begin{equation}
\label{eq:20201007e}
(\forall\, v_i \in \range(A_i^*))\quad y + v_i \in \inte(\dom \Leg^*)\ \implies\ D_C(x_i) \leq D_{\Leg^*}(y + v_i, y_\star).
\end{equation}
Indeed, Lemma~\ref{lm:aff_proj_1} yields $D_C(x_i) = \Bre(P_C(x_i), x_i) = \Bre(P_C(x),x_i)= \Bre(x_\star,P_{C_i}(x))$,
with $x_\star \in C_i$. Hence, using Lemma~\ref{lm:main1} and Fact~\ref{prop:bregman_div}\ref{prop:bregman_div_vi},
\eqref{eq:20201007e} follows.

\emph{Step }$2$: There exists $\tilde{w} \in \range(A^*)$ such that $\norm{H \tilde{w}} = 1$ and, for all $\tau>0$, 
\begin{equation}
\label{eq:20201123b}
u_\tau:= H(y_\star - y + \tau \tilde w) \in V(x_\star)\!\setminus\!\{0\}.
\end{equation}
Indeed, first recall that $\range(H A^*)=V(x_\star)\neq \{0\}$. It follows from \eqref{eq:KKT} that 
$y_\star - y \in \range(A^*)$. Now, if $H(y_\star - y)\neq 0$ we define $\tilde{w}=(y_\star - y)/\norm{H(y_\star - y)}$
and \eqref{eq:20201123b} follows. Otherwise, since $\range(H A^*)\neq \{0\}$,  we can pick $\tilde{w} \in \range(A^*)$
such that $\norm{H \tilde{w}}=1$ and again \eqref{eq:20201123b} follows.

\emph{Step }$3$: Suppose that  $HA_i^*\neq0$. We prove that, for every $\tau>0$, there exists 
$v_{i,\tau} \in \range(A_i^*)$ such that $y+v_{i,\tau} - y_\star\not\in\Ker(H)$ and
\begin{equation}
\label{eq:20201213b}
\norm{y+v_{i,\tau} - y_\star}\leq (1+M\norm{H}^2)\norm{y_\star-y} + 3\tau M \norm{H},
\end{equation}
where $\sup_{i \in I}\norm{A_i^*(A_i H^2 A_i^*)^\dagger A_i}\leq M<+\infty$, due to \ref{eq:AS3a}.
Indeed, since $HA_i^*\neq0$, there exists $w_i \in \range(HA^*_i)$ such that $\norm{w_i}=2$. Now, note that
\begin{equation}
\label{eq:20201212a}
Q_i(x_\star) = H A_i^*[A_i H^2 A_i^*]^\dagger  A_i H
\end{equation}
and let, for every $\tau>0$,
\begin{equation}
\label{eq:20201123c}
v_{i,\tau}:=A_i^*[A_i H^2 A_i^*]^\dagger  A_i H (u_\tau +\tau  w_i)\in\range(A_i^*).
\end{equation}
Then, recalling \eqref{eq:20201123b}, \eqref{eq:20201212a}, and the fact that $w = H \tilde{w}$, we have
\begin{align*}
H(y_\star-y-v_{i,\tau}) &= H(y_\star-y) - Q_i(x_\star)(u_\tau + \tau w_i)\\
&=[I - Q_i(x_\star)]H(y_\star-y) - \tau Q_i(x_\star)(w+w_i)
\end{align*}
and, since $Q_i(x_\star)$ is the projector onto $\range(HA_i^*)$ and $w_i \in \range(HA_i^*)$, we have
\begin{equation}
\label{eq:20201123a}
 \norm{H(y+v_{i,\tau} - y_\star)}^2 = \norm{[I - Q_i(x_\star)]H(y_\star-y)}^2 + \tau^2 \norm{Q_i(x_\star)w +  w_i}^2.
 \end{equation} 
 In the above formula we have $Q_i(x_\star)w \neq - w_i$, since $\norm{w_i} = 2$ while  
 $\norm{Q_i(x_\star)w} \leq \norm{w} = 1$. Therefore  $\norm{H(y+v_{i,\tau} - y_\star)}^2>0$ and hence 
 $y+v_{i,\tau} - y_\star\not\in\Ker(H)$. Finally, inequality \eqref{eq:20201213b} follows by bounding 
 $\norm{v_{i,\tau}}$ using \eqref{eq:20201123c}, \eqref{eq:20201123b}, assumption~\ref{eq:AS3a},
and the fact that $\norm{w_i} = 2$ and $\norm{H \tilde{w}}=1$. 
  
\emph{Step }$4$: Suppose that  $HA_i^*\neq0$ and let $\varepsilon \in \left]0,1\right[$. We prove that for 
$\tau>0$ sufficiently small
\begin{equation}
\label{eq:20201213c}
\frac{1 - \varepsilon}{2} \norm{u_\tau}^2 \leq D_C(x)\ \text{and}\ D_C(x_i) \leq 
\frac{1 + \varepsilon}{2} \Big(\norm{[I - Q_i(x_\star)]u_\tau}^2 + \sqrt{\tau} D_C(x) \Big). 
\end{equation}
Indeed, it follows from the second part of Fact~\ref{prop:bregman_div}\ref{prop:bregman_div_vii}, applied to 
$D_{\Leg^*}$, that there exists $\tilde \delta>0$ such that if $\norm{\tilde y-y_\star}<\tilde\delta$ and 
$\tilde y-y_\star\not\in\Ker(H)$, then 
\begin{equation*}
\frac{\sqrt{1 - \varepsilon}}{2} \scalarp{H^2 (\tilde y - y_\star), \tilde y - y_\star} \leq D_{\Leg^*}(\tilde y, y_\star) \leq \frac{1 + \varepsilon}{2} \scalarp{H^2 (\tilde y - y_\star), \tilde y - y_\star}.
\end{equation*}
Therefore, setting $\beta_\star= 1 + \max\{3 M\norm{H}^2+M\norm{H},\norm{\tilde w}\}>1$, it follows from the 
inequality $\norm{y_\star-y-\tau\tilde w} \leq \norm{y_\star-y} + \tau \norm{\tilde w}$ and \eqref{eq:20201213b} 
that if $\tau \leq \norm{y-y_\star}$ and $\norm{y-y_\star}\leq\tilde\delta / \beta_\star$, we have
\begin{equation}
\label{eq:20201005c}
D_{\Leg^*}(y+v_{i,\tau}, y_\star) 
\leq \frac{1 + \varepsilon}{2} \scalarp{H^2 (y+v_{i,\tau} - y_\star), y+v_{i,\tau} - y_\star}
\end{equation}
and
\begin{equation}\label{eq:20201005d}
D_{\Leg^*}(y+\tau \tilde w, y_\star) \geq \frac{\sqrt{1 - \varepsilon}}{2} 
\scalarp{H^2 (y +\tau \tilde w - y_\star), y +\tau \tilde w - y_\star}.
\end{equation}
Now, the continuity of $\nabla\Leg$ and Fact \ref{prop:bregman_div}\ref{prop:bregman_div_x} yields that there 
exists $\delta>0$ such that if $D_C(x)<\delta$ then, $\norm{y-y_\star}<\tilde\delta / \beta_\star$, and hence, 
collecting \eqref{eq:20201007e} and \eqref{eq:20201005c}, we obtain 
$D_C(x_i) \leq D_{\Leg^*}\big(y+v_{i,\tau},y_\star \big)\leq ((1 + \varepsilon)/2) \norm{ H(y+v_{i,\tau} - y_\star)}^2$.
However, it also holds that $\norm{ H(y+v_{i,\tau} - y_\star)} 
= \norm{[I-Q_i(x_\star)]u_\tau + \tau (w-w_i)}\leq \norm{[I-Q_i(x_\star)]u_\tau}+3 \tau$. Therefore, since 
$\norm{u_\tau} \leq \norm{H(y-y_\star)} + \tau \norm{H \tilde w} \leq (\norm{H} + 1) \norm{y - y_\star}$,
\begin{align*}
D_C(x_i) &\leq \frac{1 + \varepsilon}{2}\left( \norm{[I-Q_i(x_\star)]u_\tau}^2 
+ 9 \tau^2 +6 \tau \norm{u_\tau} 
\right)\\
&\leq \frac{1 + \varepsilon}{2}\left( \norm{[I-Q_i(x_\star)]u_\tau}^2 + 3 \tau(2 \norm{H} + 5)
\norm{y - y_\star} \right)
\end{align*}
which, for 
$\tau\leq\tau_\star^{(1)}:= \min\{\norm{y_\star-y},9^{-1}D_C(x)^2\norm{y_\star-y}^{-2}(2 \norm{H} + 5)^{-2}\}$, gives 
\begin{equation}
\label{eq:20201123d}
D_C(x_i) \leq \frac{1 + \varepsilon}{2} \Big( \norm{[I - Q_i(x_\star)]u_\tau}^2 + \sqrt{\tau} D_C(x) \Big). 
\end{equation}
On the other hand, $D_C(x) = D_{\Leg}(x_\star, x) = D_{\Leg^*}(y, y_\star)  \neq0$. Hence, using the continuity of
$D_{\Leg^*}(\cdot,y_\star)$,  we have that there exists $\tau_\star^{(2)}>0$ such that for every $\tau\leq\tau_\star^{(2)}$,
$D_C(x)\geq \sqrt{1-\varepsilon} D_{\Leg^*}(y+\tau \tilde w, y_\star)$. So, \eqref{eq:20201005d} yields
\begin{equation}
\label{eq:20201123e}
D_C(x)\geq \frac{1 - \varepsilon}{2} \scalarp{H^2 (y+ \tau \tilde{w} - y_\star), y + \tau \tilde{w}- y_\star} = \frac{1 - \varepsilon}{2} \norm{u_\tau}^2. 
\end{equation}

\emph{Step }$5$: For $\tau\leq\min\{\tau_\star^{(1)},\tau_\star^{(2)}\}$, we have
\begin{equation}\label{eq:20201120}
\frac{D_C(x_i)}{D_C(x)} \leq \frac{1+\varepsilon}{1-\varepsilon} 
\bigg( 1 - \frac{\norm{Q_i(x_\star) u_\tau}^2}{\norm{u_\tau}^2} \bigg)+ \frac{1 + \varepsilon}{2}\sqrt{\tau}.
\end{equation}
This follows from \eqref{eq:20201213c} when $HA_i^*\neq0$. However, \eqref{eq:20201120} holds even when 
$\range(HA_i^*)=\{0\}$. Indeed in such case, recalling the definition of $Q_i(x_\star)$, we have $Q_i(x_\star)\equiv 0$.
Hence, since $D_C(x_i)\leq D_C(x)$, we have that \eqref{eq:20201120} actually holds for every $\tau>0$.

\emph{Step }$6$: 
Note that inequality \eqref{eq:20201120} holds for all $i\in I$ and every $\tau\leq\min\{\tau_\star^{(1)},\tau_\star^{(2)}\}$.
Consequently, taking the infimum on both sides of \eqref{eq:20201120} we have
\begin{equation*}
\inf_{i\in I}\frac{D_C(x_i)}{D_C(x)} \leq \frac{1+\varepsilon}{1-\varepsilon} 
\bigg( 1 - \sup_{i \in I}\frac{\norm{Q_i(x_\star) u_\tau}^2}{\norm{u_\tau}^2} \bigg) + \frac{1 + \varepsilon}{2} \sqrt{\tau}.
\end{equation*}
Then \eqref{eq:20200930a} follows by recalling that $u_\tau \in V(x_\star)\!\setminus\!\{0\}$, and letting $\tau\to0$.
\end{proof}

\begin{lemma}\label{lm:sigma_I}
Under the same assumptions of Lemma~\ref{lm:gamma_I}, define the function 
$\sigma_{\C}\colon \inte(\dom\Leg)\to[0,1]$ such that, for every $x\in\inte(\dom\Leg)$,
\begin{equation}\label{eq:sigma_I}
\sigma_{\C}(x) = \begin{cases}
\displaystyle{\sup_{z\in K(x)\setminus C}}\left[ \inf_{i\in I}\frac{D_C(P_{C_i}(z))}{D_{C}(z)}\right]& \text{if } x\not\in C, \\
1-\gamma_{\C}(x) &  \text{if } x\in C,
\end{cases}
\end{equation}
where $K(x):=\left\{z\in\inte(\dom \Leg)\,\vert\, P_C(z)=P_C(x),\; D_C(z)\leq D_C(x) \right\}$. Then, 
\begin{enumerate}[label={\rm (\roman*)}] 
\item\label{lm:sigma_Ii} 
$(\forall\,x\in\inte(\dom\Leg))$
$\sigma_{\C}(x)<1$ and 
$\inf_{i\in I}D_C(P_{C_i}(x))\leq\sigma_{\C}(x) D_C(x)$;
\item\label{lm:sigma_Iii} 
$(\forall\, x,y\in\inte(\dom\Leg)\!\setminus\!C)$ $y \in K(x)\ \Rightarrow\ \sigma_{\C}(y)\leq\sigma_{\C}(x)$.
\end{enumerate}  
\end{lemma}

\begin{proof}
\ref{lm:sigma_Ii}: Clearly, by definition of $\sigma_{\C}$, for every $x\not\in C$ (since $x \in K(x)$), we have 
$\inf_{i\in I}D_C(P_{C_i}(x))\leq\sigma_{\C}(x) D_C(x)$, while for $x\in C$, the inequality is satisfied trivially. 
Next, let $x\in\inte(\dom\Leg)$ and let $x_\star = P_C(x)$. We prove that $\sigma_{\C}(x)<1$. According to
Lemma~\ref{lm:gamma_I} and definition \eqref{eq:sigma_I}, if $x\in C$, then $\sigma_{\C}(x)<1$. So, in the 
rest of the proof we assume $x\not\in C$. Let $z\in K(x)\setminus C$. Then $D_C(z)>0$ and, since 
$C = \cap_{i \in I} C_i$, there exists $j\in I$, such that $D_{C_j}(z)>0$. Now, let $\gamma_z \in \left]0,1\right]$ 
be such that $D_{C_j}(z) \geq \gamma_z D_C(z)$. Since, in virtue of 
Fact~\ref{prop:projproperties}\ref{prop:projproperties_iv}, $D_{C_j}$ and $D_C$ are continuous, there exists 
an open set $U_z\subset \inte(\dom\Leg)$ such that $z\in U_z$ and for all $\tilde z\in U_z$, 
$D_{C_j}(\tilde z) \geq \gamma_z D_C(\tilde z)$. Then, since $C \subset C_j$, by Lemma~\ref{lm:aff_proj_1}, 
we have $D_C(P_{C_j}(\tilde z))= D_C(\tilde z) - D_{C_j}(\tilde z)\leq (1-\gamma_z) D_C(\tilde z)$, and hence 
\begin{equation}
\label{eq:20201001a}
(\forall\, \tilde{z} \in U_z)\quad
\inf_{i\in I}D_{C}(P_{C_i}(\tilde z))\leq (1-\gamma_z) D_C(\tilde z),\quad\text{with}\ \ 1-\gamma_z<1.
\end{equation}
Next, let  $z\in K(x)\cap C$ and take $\varepsilon>0$ such that 
$(1+\varepsilon)(1-\varepsilon)^{-1}[1-\gamma_{\C}(x_\star)]<1$. Moreover, in virtue of 
Lemma \ref{lm:gamma_I}, there exists $\delta>0$ such that \eqref{eq:20200930a} holds. So, define 
$U_z:=\{ \tilde z \in \inte(\dom\Leg)\,\vert\, \Bre(z,\tilde z)<\delta \}$. Then,  for $\tilde z\in U_z$, since $z\in C$, 
we have that $D_C(\tilde z)\leq \Bre(z,\tilde z)<\delta$, and hence, by \eqref{eq:20200930a},
\begin{equation}
\label{eq:20201001b}
\inf_{i\in I} D_{C}(P_{C_i}(\tilde z))\leq \frac{1+\varepsilon}{1-\varepsilon}[1-\gamma_{\C}(P_C(\tilde z))]  D_C(\tilde z).
\end{equation}
Now, the set $\tilde K(x) = \left\{z\in\inte(\dom \Leg)\,\vert\,\; \Bre(P_C(x),z)\leq D_C(x) \right\}$ is 
compact by Fact \ref{prop:bregman_div}\ref{prop:bregman_div_ix}. Moreover, by 
Fact~\ref{prop:bregman_proj}\ref{prop:bregman_proj_0}, 
$K(x) =\{z\in\tilde K(x)\,\vert\, \nabla\Leg(z)-\nabla\Leg(x)\in\range(A^*)\}$. Thus, it follows from the continuity 
of $\nabla\Leg$ that $K(x)$ is a closed subset of the compact set $\tilde K(x)$, and hence it is compact. Thus, 
since $\bigcup_{z\in K(x)} U_z$ is an open covering of $K(x)$, there exist points $z_1,\ldots,z_m\in K(x)$ 
so that  $K(x)\subset U_{z_1} \cup\dots\cup U_{z_m}$, and we can define  
\begin{equation}
\sigma = \max\big\{1 - \gamma_{z_1}, \dots, 1 - \gamma_{z_m}, (1+\varepsilon)(1-\varepsilon)^{-1}[1-\gamma_{\C}(P_C(x_\star))\big\}<1.
\end{equation}
We let $z\in K(x)\setminus C$. Then, there exists $j\in\{1,\ldots,m\}$, such that $z\in U_{z_j}$. If $z_j\not\in C$, 
then we derive from \eqref{eq:20201001a} that $\inf_{i\in I} D_{C}(P_{C_i}(z))/D_C(z)\leq 1-\gamma_{z_j}\leq \sigma$,
while if $z_j\in C$, since $P_C(z)=x_\star$, by \eqref{eq:20201001b} we have 
$\inf_{i\in I} D_{C}(P_{C_i}(z))/D_C( z)\leq (1+\varepsilon)(1-\varepsilon)^{-1}[1-\gamma_{\C}(x_\star)]\leq \sigma$.
Therefore, 
\begin{equation}
\sigma_{\C}(x)=\displaystyle{\sup_{z\in K(x)\setminus C}}\left[ \inf_{i\in I}\frac{D_C(P_{C_i}(z))}{D_{C}(z)}\right]\leq\sigma<1.
\end{equation}

\ref{lm:sigma_Iii}: 
Let $x,y\in\inte(\dom\Leg)\setminus C$ be such that $P_C(y)=P_C(x)$ and $D_C(y)\leq D_C(x)$. Then, clearly 
$K(y) \subset K(x)$ and hence $K(y)\setminus C\subset K(x)\setminus C$. Then, since $x,y \notin C$, the 
statement follows from the definition of $\sigma_{\C}$ in \eqref{eq:sigma_I}.
\end{proof}

\begin{theorem}[Greedy set control scheme]
\label{thm:main2}
With reference to Problem~\ref{P1}, suppose that \ref{eq:stand1}, \ref{eq:AS3}, and \ref{eq:AS3a} hold. 
Let $(x_k)_{k \in \N}$ be generated by Algorithm~\ref{algoABP} with the control \ref{eq:C1}  and
let $x_\star = P_{C}(x_0)$. Then the following hold.
\begin{enumerate}[label={\rm (\roman*)}] 
\item\label{thm:main2_0} $(\forall\, k \in \N)\quad D_C(x_{k+1}) \leq \sigma_{\C}(x_{k})\, D_C(x_{k})$.
\item\label{thm:main2_i}  
$(\forall\, n \in \N)(\forall\, k \in \N)\quad D_C(x_{k+1+n})\leq \sigma_{\C}(x_n) D_C(x_{k+n})$.
\item\label{thm:main2_iii} 
Either $D_C(x_k) \to 0$ in a finite number of iterations or the sequence $(\sigma_\C(x_k))_{k \in \N}$ is decreasing and
\begin{equation*}
\displaystyle{\limsup_{k\to\infty}\frac{D_C(x_{k+1})}{D_C(x_{k})} \leq \lim_{k\to\infty}\sigma_{\C}(x_{k}) \leq \sigma_{\C}(x_\star)}.\end{equation*}
\end{enumerate}
\end{theorem}
\begin{proof}
\ref{thm:main2_0}:
Let $k \in \N$. By Theorem \ref{thm3}\ref{thm3_i}, we have that 
\begin{equation}
\label{eq:20201007g}
D_C(x_{k+1}) = D_C(x_{k})-D_{C_{\xi_k}}(x_{k}), 
\end{equation}
where the greedy choice ensures that $D_{C_{\xi_{k}}}(x_{k})=\sup_{i\in I} D_{C_{i}}(x_{k})$. Thus, 
$D_C(x_{k+1}) = \inf_{i\in I} (D_C(x_{k})-D_{C_{i}}(x_{k}))$ and hence, by Lemma \ref{lm:aff_proj_1},
\begin{equation}
D_C(x_{k+1})= \inf_{i\in I}D_C(P_{C_i}(x_{k})). 
\end{equation}
So, the statement follows from Lemma~\ref{lm:sigma_I}\ref{lm:sigma_Ii}.

\ref{thm:main2_i}: 
First note that Theorem~\ref{thm3}\ref{thm3_i} yields that $(D_C(x_k))_{k \in \N}$ is decreasing and that 
$(P_{C}(x_{k}))_{k \in \N} \equiv x_\star$. Let $n,k \in \N$. If $D_C(x_{k+1+n})=0$, then the inequality holds.
Suppose that $D_C(x_{k+1+n})>0$. Then, $0<D_C(x_{k+1+n}) \leq D_C(x_{k+n}) \leq D_C(x_n)$
and hence $x_n, x_{k+n} \notin C$ and, using the notation of Lemma \ref{lm:sigma_I}, $x_{k+n} \in K(x_n)$. 
So, by Lemma \ref{lm:sigma_I}\ref{lm:sigma_Iii}, we have $\sigma_\C(x_{k+n}) \leq \sigma_\C(x_n)$ and the 
statement follows from \ref{thm:main2_0}.

\ref{thm:main2_iii}: 
If $D_C(x_n)=0$ for some $n \in \N$, since $(D_C(x_k))_{k \in \N}$ is decreasing, then, for every integer 
$k \geq n$, $D_C(x_k)=0$ and hence $D_C(x_k) \to 0$ in a finite number of iterations. Suppose that for every 
$k \in \N$, $D_C(x_k)>0$, that is, $x_k \notin C$. Then $x_{k+1} \in K(x_k)$ and hence, by 
Lemma \ref{lm:sigma_I}\ref{lm:sigma_Iii}, $\sigma_{\C}(x_{k+1}) \leq \sigma_\C(x_k)$. Moreover,
according to Lemma \ref{lm:gamma_I}, given $\varepsilon\in\left]0,1\right[$ there exists $\delta>0$ be such 
that \eqref{eq:20200930a} holds. Since, by \ref{thm:main2_i}, $D_C(x_k)\to 0$, there exists $n \in \N$ such 
that for all integer $k \geq n$, $D_C(x_{k})<\delta$, and hence for every $z \in K(x_k)\setminus C$, since 
$D_C(z) \leq D_C(x_k) \leq \delta$, \eqref{eq:20200930a} yields
\begin{equation*}
\inf_{i\in I} \frac{D_C(P_{C_i}(z))}{D_{C}(z)}\leq \frac{1+\varepsilon}{1-\varepsilon} [1-\gamma_{\C}(x_\star)]
\end{equation*}
and, recalling \eqref{eq:sigma_I}, $\sigma_{\C}(x_{k})\leq (1+\varepsilon)(1-\varepsilon)^{-1}\sigma_{\C}(x_\star)$. 
Therefore, using also \ref{thm:main2_0},
\begin{equation}
\displaystyle{\limsup_{k\to\infty}\frac{D_C(x_{k+1})}{D_C(x_{k})} \leq \lim_{k\to\infty}\sigma_{\C}(x_{k}) = \inf_{k\in \N}\sigma_{\C}(x_{k}) \leq \frac{1+\varepsilon}{1-\varepsilon}\sigma_{\C}(x_\star)}.
\end{equation}
Since $\varepsilon$ is arbitrary in $\left]0,1\right[$, the statement follows.
\end{proof}

\subsection{Random set control scheme}
We now address Problem~\ref{P2}. Following the same line of presentation as in the greedy case, we first 
give a general theorem of convergence and then we analyze the rate of convergence.

\begin{theorem}
\label{thm3b}
With reference to Problem~\ref{P2}, suppose that \ref{eq:stand1} and \ref{eq:AS3} hold. Let $(x_k)_{k \in \N}$ 
be generated by Algorithm~\ref{algoABP} using the random set control scheme \ref{eq:C2}  and 
let $x_\star = P_{C}(x_0)$.
Then the following hold.

\begin{enumerate}[label={\rm (\roman*)}]
\item\label{prop_randomaffine_i}
$(\forall\, k \in \N)\ \nabla \Leg(x_\star) - \nabla \Leg(x_k) \in \range(A^*)$ $\PP$-a.s.
\item\label{prop_randomaffine_ii}
$(\forall\, k \in \N)\ P_C(x_k)= x_\star\ $ $\PP$-a.s.~and $\ D_C(x_{k+1}) + D_{C_{\xi_k}}(x_{k}) = D_C(x_{k})$ 
$\PP$-a.s.
\item\label{prop_randomaffine_iii}  
Under the assumptions of Theorem~\ref{thm:conv2}, $x_k \to\! x_\star$ $\PP$-a.s.~and 
$\EE[\norm{x_k \!-\! x_\star}^2]\to \!0$.
\end{enumerate}
\end{theorem}

\begin{proof}
It follows from Proposition~\ref{prop:stochasticintersection2}\ref{prop:stochasticintersection2_ii} that 
there exists a $\PP$-negligible set $N \subset \Omega$ such that, for every $k \in \N$,
$C = \bigcap_{\omega \in \Omega\setminus N} C_{\xi_k(\omega)}$. 
Let $\omega \in \Omega\setminus N$ and $k \in \N$. 

\ref{prop_randomaffine_i}:
Since $x_{k+1}(\omega) = P_{C_{\xi_k(\omega)}}(x_{k}(\omega))$
and $C \subset C_{\xi_k(\omega)}$,  $\range(A^*_{\xi_k(\omega)}) \subset \range(A^*)$
and, proceeding as in the proof of Theorem~\ref{thm3}\ref{thm3_0}, one gets
$\nabla \Leg(x_\star) - \nabla \Leg(x_k(\omega)) \in \range(A^*)$.

\ref{prop_randomaffine_ii}: 
As in the proof of Theorem~\ref{thm3}\ref{thm3_i},
since $x_{k+1}(\omega) = P_{C_{\xi_k(\omega)}}(x_{k}(\omega))$ and $C \subset C_{\xi_k(\omega)}$,
it follows from Lemma~\ref{lm:aff_proj_1} that
\begin{equation*}
P_C(x_{k}(\omega)) = P_C(x_{k+1}(\omega))\ \text{and}\ 
D_C(x_{k+1}(\omega)) + D_{C_{\xi_k(\omega)}}(x_{k}(\omega)) = D_C(x_{k}(\omega)).
\end{equation*}
The statement follows.

\ref{prop_randomaffine_iii}:
By Theorem~\ref{thm:conv2} we have $x_k \to \hat{x}$ $\PP$-a.s.~for some 
$\hat{x} \in C\cap \inte(\dom \Leg)$ $\PP$-a.s. Since $P_C$ is continuous on $\inte(\dom \Leg)$, 
we have that $P_C(x_k) \to P_C(\hat{x}) = \hat{x}$ $\PP$-a.s. However, item \ref{thm3_i} yields that 
$P_C(x_k) = x_\star$ $\PP$-a.s., for every $k \in \N$. Therefore, $\hat{x} =  x_\star$ $\PP$-a.s. 
Finally, convergence in mean square follows by Fact~\ref{fact_convRV}.
\end{proof}

As before, we give two technical lemmas followed by the theorem on linear convergence.

\begin{lemma}\label{lm:gamma_mu}
Under the assumptions of Problem~\ref{P2}, suppose  that \ref{eq:stand1}, \ref{eq:AS3}, and \ref{eq:AS3b} hold. 
Let, for every $x\in\inte(\dom\Leg)\cap C$, $Q_i(x)$ be the orthogonal projection onto 
$\range([\nabla^2 \Leg^*(\nabla\Leg(x))]^{1/2} A_i^*)$, $V(x)=\range([\nabla^2 \Leg^*(\nabla\Leg(x))]^{1/2} A^*)$, 
let $\overline{Q}(x) := \EE[Q_\xi(x)]$, and let  
\begin{equation}
 \label{eq:main2_mu}
 \gamma_{\C,\mu}(x) 
 = \inf_{v\in V(x)\setminus\{0\}}
 \frac{\scalarp{\overline{Q}(x) v,v}}{ \norm{v}^2}.
\end{equation}

Then the following hold.

\begin{enumerate}[label={\rm (\roman*)}] 
\item\label{lm:gamma_mu_i} 
For every $x\in\inte(\dom\Leg)\cap C$, $V(x)=\Ker(\overline{Q}(x))^\perp$, hence, 
$\gamma_{\C,\mu}(x)\in\left]0,1\right]$ is the smallest nonzero eigenvalue of $\overline{Q}(x)$. 
\item\label{lm:gamma_mu_ii} 
For every $\varepsilon\in\left]0,1\right[$ there exists $\delta>0$ such that for every $x\in\inte(\dom\Leg)$
\begin{equation}\label{eq:20201005b}
D_C(x)< \delta \implies\ \EE[D_C(P_{C_\xi}(x))] 
\leq \frac{1+\varepsilon}{1-\varepsilon} [1-\gamma_{\C,\mu}(P_C(x))] \,D_{C}(x).
\end{equation}
\end{enumerate}
\end{lemma}
\begin{proof}
\ref{lm:gamma_mu_i}: 
Let $x\in\inte(\dom\Leg)\cap C$ and set $y=\nabla\Leg(x)$ and $H = [\nabla^2 \Leg^*(y)]^{1/2}$. 
We prove that $\Ker(\overline{Q}(x)) =\range(H A^*)^\perp$. Since $\scalarp{\overline{Q}(x)v,v} = 0$ 
$\ \Leftrightarrow\ \scalarp{Q_\xi(x)v,v}=0$ $\PP$-a.s., we have 
\begin{equation}
v \in \Ker(\overline{Q}(x))\ \Leftrightarrow\ v \in \Ker(Q_\xi(x))\ \PP\text{-a.s.}
\end{equation}
Moreover, $\Ker(Q_i(x)) = \range(Q_i(x))^\perp = \range(H A_i^*)^\perp = \Ker(A_i H)$. Therefore, 
$v \in \Ker(\overline{Q}(x))\ \Leftrightarrow\ H v \in \Ker(A_\xi)$ $\PP$-a.s. Now, we observe that, since 
$x \in C$, we have $x \in C_\xi$ $\PP$-a.s., and, hence $C = \Ker(A) + x$ and $C_\xi = \Ker(A_\xi) + x$ 
$\PP$-a.s.~Thus, recalling the definition of $C$ in Problem~\ref{P2},
\begin{equation*}
(\forall\, u \in X)\ u \in \Ker(A)\ \Leftrightarrow\ u+x \in C \ \Leftrightarrow\ u+x \in C_\xi\ \PP\text{-a.s.}
\ \Leftrightarrow\ u \in \Ker(A_\xi)\ \PP\text{-a.s.}
\end{equation*}
Therefore, $v \in \Ker(\overline{Q}(x))\ \Leftrightarrow\ H v \in \Ker(A_\xi)$ $\PP$-a.s. $\Leftrightarrow$ 
$H v \in \Ker(A)\ \Leftrightarrow\ v \in \Ker(A H) = \range(H A^*)^\perp$. The first part of the statement follows.
Now, in view of what we have just proved, we note that
\begin{equation}
\gamma_{\C,\mu}(x) = \min_{\substack{v\in\Ker(\overline{Q}(x))^\perp\\ \norm{v}= 1}}
 \scalarp{\overline{Q}(x) v,v}.
\end{equation}
Since, for every $i\in I$, $Q_i(x)$ is self-adjoint and positive, it follows from linearity of expectation that 
$\overline{Q}(x)$ is self-adjoint and positive, too. Hence, in virtue of Fact \ref{fact_inf}, $\gamma_{\C,\mu}(x)>0$ 
is smallest positive eigenvalue of $\overline{Q}(x)$.

\ref{lm:gamma_mu_ii}:
It follows from Proposition~\ref{prop:stochasticintersection2}\ref{prop:stochasticintersection2_ii}
that there exists a $\PP$-negligible set $N \subset \Omega$ such that 
$C = \bigcap_{\omega \in \Omega\setminus N} C_{\xi(\omega)}$.
Note that, if $x\in\inte(\dom\Leg)\cap C$, then $0=D_C(x)=D_C(P_{C_{\xi(\omega)}}(x))$, for every 
$\omega \in \Omega\setminus N$, hence \eqref{eq:20201005b} holds trivially. Therefore, we let 
$x\in\inte(\dom\Leg) \setminus C$ and let $x_\star=P_C(x)$, $y = \nabla\Leg(x)$, $y_\star = \nabla \Leg(x_\star)$. 
Now, let $\omega \in \Omega\setminus N$ and set $x_{\xi(\omega)} = P_{C_{\xi(\omega)}}(x)$ and 
$y_{\xi(\omega)} = \nabla \Leg(x_{\xi(\omega)})$. Then, following  the same reasoning as in Lemma \ref{lm:gamma_I}, 
we obtain (similarly to \eqref{eq:20201120}) that for every $\varepsilon\in\left]0,1\right[$ there exists $\delta>0$ 
such that if $D_C(x)<\delta$, then 
\begin{equation*}
\frac{D_C(P_{C_{\xi(\omega)}}(x))}{D_C(x)} \leq \frac{1+\varepsilon}{1-\varepsilon} 
\bigg( 1 - \frac{\norm{Q_{\xi(\omega)}(x_\star) u_\tau}^2}{\norm{u_\tau}^2} \bigg)+ \frac{1+\varepsilon}{2}\sqrt{\tau},
\end{equation*}
where $u_\tau = [\nabla^2 \Leg^*(y_\star)]^{1/2}(y_\star - y + \tau \tilde{w}) \in V(x_\star) \setminus\{0\}$ and $\tau>0$  
is small enough and independent on $\xi(\omega)$. The above inequality implies that 
\begin{equation*}
\frac{D_C(P_{C_\xi}(x))}{D_C(x)} \leq \frac{1+\varepsilon}{1-\varepsilon} 
\bigg( 1 - \frac{\norm{Q_\xi(x_\star) u_\tau}^2}{\norm{u_\tau}^2} \bigg)+ \frac{1+\varepsilon}{2}\sqrt{\tau},\;\PP\text{-a.s.}
\end{equation*}
Hence, taking the expectation and recalling definition \eqref{eq:main2_mu}, we have
\begin{equation*}
\frac{\EE[D_C(P_{C_\xi}(x))]}{D_C(x)}
\leq \frac{1+\varepsilon}{1-\varepsilon} 
\bigg( 1 - \frac{\norm{\overline{Q}(x_\star) u}^2}{\norm{u}^2} \bigg) 
+ \frac{1+\varepsilon}{2}\sqrt{\tau} \leq 
\frac{1+\varepsilon}{1-\varepsilon} 
[1 - \gamma_{\C}(x_\star)]+ \frac{1+\varepsilon}{2}\sqrt{\tau}.
\end{equation*}
Finally, letting $\tau\to0$ in the above inequality the statement follows.
\end{proof}

\begin{lemma}\label{lm:sigma_mu}
Under the same assumptions of Lemma~\ref{lm:gamma_mu}, define the function
$\sigma_{\C,\mu}\colon \inte(\dom\Leg)\to[0,1]$ such that
\begin{equation}\label{eq:sigma_mu}
\sigma_{\C,\mu}(x) = \begin{cases}
\displaystyle{\sup_{z\in K(x)\setminus C}}\left[\frac{\EE[D_C(P_{C_\xi}(z))]}{D_{C}(z)}\right]& \text{ if }x\not\in C, \\
1-\gamma_{\C,\mu}(x) & \text{ if } x\in C,
\end{cases}
\end{equation}
where, $K(x)$ is defined as in Lemma~\ref{lm:sigma_I}. Then, the following hold.

\begin{enumerate}[label={\rm (\roman*)}] 
\item\label{lm:sigma_mu_i} 
Suppose that $C = \bigcap_{i \in I} C_i$. Then, for every $x\in\inte(\dom\Leg)$, $\sigma_{\C}(x)\leq \sigma_{\C,\mu}(x)$.
\item\label{lm:sigma_mu_ib} 
For every $\varepsilon \in \left]0,1\right[$ there exists $\delta>0$ such that
\begin{equation*}
(\forall\, x\in\inte(\dom\Leg))\ D_C(x) 
\leq \delta\ \Rightarrow\ \sigma_{\C,\mu}(x) \leq \frac{1+\varepsilon}{1-\varepsilon} 
\big(1 - \gamma_{\C,\mu}(P_C(x))\big).
\end{equation*}
\item\label{lm:sigma_mu_ii} 
$(\forall\, x\in\inte(\dom\Leg))$
$\sigma_{\C,\mu}(x)<1\ \ \text{and}\ \ \EE[D_C(P_{C_\xi}(x))]\leq\sigma_{\C,\mu}(x) \,D_C(x)$. 
\item\label{lm:sigma_mu_iii} 
$(\forall\, x,y\in\inte(\dom\Leg)\!\setminus\!C)$ $y \in K(x)\ \Rightarrow\ \sigma_{\C,\mu}(y)\leq\sigma_{\C,\mu}(x)$.
\end{enumerate}  
\end{lemma}

\begin{proof}
\ref{lm:sigma_mu_i}: 
Comparing \eqref{eq:sigma_mu} with \eqref{eq:sigma_I}, the statement is clear if $x\not\in C$. 
On the other hand, for $x\in C$, recalling \eqref{eq:main2_i} and \eqref{eq:main2_mu}, the statement 
follows from the fact that $\scalarp{\EE[Q_\xi(x)]v,v} = \EE[\scalarp{Q_\xi(x)v,v}] 
\leq \sup_{i\in I} \scalarp{Q_i(x)v,v} = \sup_{i\in I} \norm{Q_i(x)v}^2$.

\ref{lm:sigma_mu_ib}:
Let $\varepsilon$ and $\delta$ as in Lemma~\ref{lm:gamma_mu}\ref{lm:gamma_mu_ii}
and $x \in \inte(\dom\Leg)$ such that $D_C(x) \leq \delta$. If $x \in C$, then, by definition, 
$\sigma_{\C,\mu}(x) = 1 - \gamma_{\C,\mu}(x)$ and the statement holds. Suppose that $x \notin C$.
Then, for every $z \in K(x)\setminus\! C$, we have $D_C(z) \leq D_C(x) \leq \delta$ and hence,
by \eqref{eq:20201005b}, $\EE[D_C(P_{C_\xi}(z))]/D_C(z) \leq (1+\varepsilon)(1-\varepsilon)^{-1} 
(1-\gamma_{\C,\mu}(P_C(x)))$. The statement follows from the definition of 
$\sigma_{\C,\mu}$ in \eqref{eq:sigma_mu}.
	
\ref{lm:sigma_mu_ii}:
We proceed as in the proof of Lemma \ref{lm:sigma_I}\ref{lm:sigma_Ii}. 
Let $x\in\inte(\dom\Leg)$ and set $x_\star = P_C(x)$. According to Lemma \ref{lm:gamma_mu}, 
if $x\in C$, then $\sigma_{\C,\mu}(x)<1$. So, assume $x\not\in C$. Then, for any $z\in K(x) \setminus C$, 
$D_C(z)>0$ and, according to Proposition \ref{prop:stochasticintersection2}, $\overline{D}_{C}(z)>0$. 
Thus, there exists $0<\gamma_z\leq 1$ such that $\overline{D}_C(z)-\gamma_z D_C(z)>0$, but, 
since $D_C$ is continuous and $\overline{D}_C$ is lower semicontinuous, there exists an open set 
$U_z\subset\inte(\dom\Leg)$ such that $z\in U_z$ and for all $\tilde z\in U_z$, 
$\overline{D}_{C}(\tilde z) \geq \gamma_z D_C(\tilde z)$. Using again Lemma \ref{lm:aff_proj_1},
we have that $\EE[D_C(P_{C_\xi}(\tilde z))]\leq (1-\gamma_z) D_C(\tilde z)$. 
The rest of the proof follows the same lines of argument as in Lemma \ref{lm:sigma_I}, just by replacing 
$\inf_{i\in I} D_C(P_{C_i}(\cdot))$, $\gamma_{\C}$, $\sigma_{\C}$ and Lemma \ref{lm:gamma_I} by 
$\EE[D_C(P_{C_\xi}(\cdot))]$, $\gamma_{\C,\mu}$, $\sigma_{\C,\mu}$ and Lemma \ref{lm:gamma_mu}, 
respectively. 

\ref{lm:sigma_mu_iii}: 
The proof is identical to that of Lemma~\ref{lm:sigma_I}\ref{lm:sigma_Iii},
but now uses \eqref{eq:sigma_mu}.
\end{proof}

\begin{theorem}[Random set control scheme]
\label{thm:main_random}
With reference to Problem~\ref{P2},
suppose that \ref{eq:stand1}, \ref{eq:AS3}, and \ref{eq:AS3b} hold. 
Let $(x_k)_{k \in \N}$ be generated by Algorithm~\ref{algoABP} using
the random set control scheme \ref{eq:C2}  and
let $x_\star = P_{C}(x_0)$.
Then, the following hold.

\begin{enumerate}[label={\rm (\roman*)}] 
\item\label{thm:main_random_0}  
$(\forall\, k \in \N)\quad
\EE[D_C(x_{k+1})\,\vert\,\mathfrak{X}_{k}]\leq \sigma_{\C,\mu}(x_{k})D_C(x_{k})\ \PP$-a.s.
\item\label{thm:main_random_i} 
$(\forall\, n \in \N)(\forall\, k \in \N)\quad\EE[D_C(x_{k+1+n})]\leq (\esssup\sigma_{\C,\mu}(x_{n})) \EE[ D_C(x_{k+n})]$.
\item\label{thm:main_random_iii} 
For every $\varepsilon\in \left]0,1\right[$, there exists $\delta>0$, such that if $\norm{x_n-x_\star}\leq \delta$ holds $\PP$-a.s. for some $n\in\N$, then for every $k\geq n$
\begin{equation}
\EE[D_C(x_{k+1})] \leq \frac{1+\varepsilon}{1-\varepsilon} \sigma_{\C,\mu}(x_\star)\, \EE[D_C(x_{k})].
\end{equation}
\item\label{thm:main_random_iv} 
For every $\varepsilon, \alpha \in \left]0,1\right[$, there exists $\mathcal{A} \in \mathfrak{A}$ and $n\in\N$, 
such that $\PP[\Omega\setminus \mathcal{A}]\leq\alpha$ and, for all $k\geq n$, we have
\begin{equation}
\label{eq:main2_ii}
\EE[D_C(x_{k+1})\,\vert\,\mathcal{A}] \leq \frac{1+\varepsilon}{1-\varepsilon}  
\sigma_{\C,\mu}(x_\star) \EE[D_C(x_{k})\,\vert\,\mathcal{A}].
\end{equation}
\end{enumerate}
\end{theorem}
\begin{proof}
\ref{thm:main_random_0}:
First, observe that from Theorem \ref{thm3b}\ref{prop_randomaffine_ii} for every $k\in\N$, $P_C(x_k) = x_\star$ and, therefore, using Lemma \ref{lm:sigma_mu}\ref{lm:sigma_mu_ii} and Fact~\ref{f:20190112d}, the inequality follows. 

\ref{thm:main_random_i}: Again Theorem~\ref{thm3b}\ref{prop_randomaffine_ii} yields that
$(D_C(x_k))_{k\in\N}$ is decreasing $\PP$-a.s.~and that $P_C(x_k)= x_\star$ $\PP$-a.s., for every $k \in \N$.
So, also in virtue of \ref{thm:main_random_0}, there exists a $\PP$-negligible set $N$ 
such that for every $\omega \in \Omega\setminus N$,
$(D_C(x_k(\omega)))_{k\in\N}$ is decreasing, $(P_C(x_k(\omega)))_{k \in \N} \equiv x_\star$,
and $\EE[D_C(x_{k+1})\,\vert\,\mathfrak{X}_{k}](\omega)\leq \sigma_{\C,\mu}(x_{k}(\omega))D_C(x_{k}(\omega))$.
Let $n,k \in \N$ and $\omega \in \Omega\setminus N$. We prove that 
\begin{equation}
\label{eq:20201216b}
\eta_{k+1+n}(\omega):=\EE[D_C(x_{k+1+n})\,\vert\,\mathfrak{X}_{k+n}](\omega)
\leq \sigma_{\C,\mu}(x_{n}(\omega))D_C(x_{k+n}(\omega)).
\end{equation}
Then, since the above inequality holds $\PP$-a.s., the statement will follow by just 
majorizing $\sigma_{\C,\mu}(x_{n})$ with its essential supremum and then taking expectation.
Now, if $\eta_{k+1+n}(\omega) =0$, then \eqref{eq:20201216b} holds.
Otherwise, since $\sigma_{\C,\mu}(x_{k+n}(\omega)) \leq 1$, we have 
$0< \eta_{k+1+n}(\omega) \leq \sigma_{\C,\mu}(x_{k+n}(\omega)) D_C(x_{k+n}(\omega)) 
\leq D_C(x_{k+n}(\omega)) 
\leq D_C(x_n(\omega))$.
Hence $x_{k+n}(\omega), x_n(\omega) \notin C$ 
and, using the notation of Lemma~\ref{lm:sigma_mu}, $x_{k+n}(\omega) \in K(x_n(\omega))$.
Therefore, Lemma~\ref{lm:sigma_mu}\ref{lm:sigma_mu_iii} yields $\sigma_{\C,\mu}(x_{k+n}(\omega)) 
\leq \sigma_{\C,\mu}(x_n(\omega))$
and hence $\eta_{k+1+n}(\omega) \leq \sigma_{\C,\mu}(x_n(\omega)) D_C(x_{k+n}(\omega))$, so that 
\eqref{eq:20201216b} holds.

\ref{thm:main_random_iii}: Let $\varepsilon\in \left]0,1\right[$ and $\delta>0$ be from 
Lemma \ref{lm:sigma_mu}\ref{lm:sigma_mu_ib}.  
Since $\Bre(x_\star,\cdot)$ is continuous, 
there exists $\delta_1 >0$ such that if $\norm{x_n-x_\star}\leq \delta_1$ $\PP$-a.s. for some $n\in\N$, then 
$D_C(x_n)=\Bre(x_\star, x_n)\leq \delta$ $\PP$-a.s. Then, for every integer $k \geq n$,
we have $D_C(x_k) \leq D_C(x_n) \leq \delta$ $\PP$-a.s.~and hence, by
Lemma~\ref{lm:sigma_mu}\ref{lm:sigma_mu_ib},
$\sigma_{\C,\mu}(x_k) \leq (1+\varepsilon)(1-\varepsilon)^{-1} \sigma_{\C,\mu}(x_\star)D_C(x_k)$.
The statement follows by using \ref{thm:main_random_0} and then taking the expectation.

\ref{thm:main_random_iv}:
Let $\varepsilon\in\left]0,1\right[$ and $\delta>0$ as in Lemma \ref{lm:gamma_mu}\ref{lm:gamma_mu_ii} and set 
$\mathcal{A}_{k}:=\{ D_C(x_{k})\leq\delta \}$. Then, 
denoting by $\chi_{\mathcal{A}_{k}}$ the characteristic function of the set $\mathcal{A}_{k}$ and using, 
as before, Lemma \ref{lm:sigma_mu} and \ref{thm:main_random_0},
we have
\begin{equation}
\label{eq:20200826d}
\EE[D_C(x_{k+1}) \,\vert\,\mathfrak{X}_{k}] \chi_{\mathcal{A}_{k}} \leq \frac{1 + \varepsilon}{1-\varepsilon} 
\sigma_{\C,\mu}(x_\star) D_C(x_{k}) \chi_{\mathcal{A}_{k}}.
\end{equation}
Moreover, Markov inequality yields $\PP[\Omega\setminus\mathcal{A}_{k}]\leq \EE[D_C(x_{k})]/\delta$, 
so, since, in virtue of \ref{thm:main_random_i} with $n=0$, 
$\EE[D_C(x_{k})]\to0$, for every $\alpha \in \left]0,1\right[$ there exists $n\in\N$ such that,  
$\PP(\Omega\setminus\mathcal{A}_{n})\leq\alpha$. Let $\mathcal{A}:=\mathcal{A}_{n}$. Since $(D_C(x_k))_{k\in\N}$ 
is decreasing $\PP$-a.s., then for all $k\geq n$, $\mathcal{A}\subset\mathcal{A}_{k}$, except for a negligible set. 
Hence, from \eqref{eq:20200826d} we obtain 
\begin{equation*}
\EE[D_C(x_{k+1}) \,\vert\,\mathfrak{X}_{k}] \chi_{\mathcal{A}} \leq \frac{1 + \varepsilon}{1-\varepsilon} 
\sigma_{\C,\mu}(x_\star) D_C(x_{k}) \chi_{\mathcal{A}}.
\end{equation*}
But, using $\EE[D_C(x_{k+1}) \chi_{\mathcal{A}}] 
= \EE[\EE[D_C(x_{k+1}) \chi_{\mathcal{A}} \,\vert\, \mathfrak{X}_{k}] ] 
= \EE[\EE[D_C(x_{k+1})\,\vert\, \mathfrak{X}_{k}]\, \chi_{\mathcal{A}} ]$ and $\EE[(\cdot)\, \chi_{\mathcal{A}}] 
= \EE[(\cdot)\,\vert\, \mathcal{A}] \PP(\mathcal{A})$ we obtain \eqref{eq:main2_ii}.
\end{proof}

\begin{remark}
\label{rmk:Qlin}
\normalfont
Taking $n=0$ in Theorem~\ref{thm:main_random}\ref{thm:main_random_i} we get the 
global $Q$-linear rate in expectation: $\EE[D_C(x_{k+1})]\leq \sigma_{\C,\mu}(x_{0}) \EE[ D_C(x_{k})]$,
where $\sigma_{\C,\mu}(x_{0})<1$. However, when $n>0$, even though 
 $\sigma_{\C,\mu}(x_{n})<1$ $\PP$-a.s., we are not ensured that
$\esssup \sigma_{\C,\mu}(x_{n})<1$.
\end{remark}

\subsection{Adaptive random set control scheme}
\label{sec:adaptive}

In this section we address again Problem~\ref{P2}, but as for Algorithm~\ref{algoABP},
 we do not assume anymore that the set control indexes 
 $\xi_k$'s are independent copies of $\xi$, but, instead, that they are set adaptively during the algorithm. 
 More precisely, we consider the case where the probability distribution of the next step is adapted to the information 
 that we may obtain in the present. We formalize this set control scheme in the following assumption, where 
 $\mu$ is again the distribution of $\xi$ and $\overline{D}_C$ is still defined as in \eqref{eq:20200611b}.

\begin{enumerate}[label={\rm \textbf{C3}},leftmargin=4.6ex]
\item\label{eq:C3}
 $\nu\colon \mathcal{I}\times \inte(\dom \Leg) \to \R_+$ is a probability kernel such that\footnote{
This means that $\nu(A,x) = \int_{A} [D_{C_{i}}(x)/\overline{D}_C(x)] \mu(d i)$ if $x \notin C$ and
$\nu(A,x) = \mu(A)$ if $x \in C$.}
\begin{equation}
\label{eq:20200924a}
(\forall\,x \in \inte(\dom \Leg))\qquad \nu(\cdot, x) = 
\begin{cases}
\dfrac{D_{C_{(\cdot)}}(x)}{\overline{D}_C(x)} \mu &\text{if } x \notin C\\
\mu &\text{if } x \in C.
\end{cases}
\end{equation}
Moreover, $(\xi_k)_{k \in \N}$ is a sequence of $I$-valued random variables and
$(x_k)_{k \in \N}$ is a sequence of $X$-valued random variables defined recursively 
according to Algorithm~\ref{algoABP}. For all $k \in \N$,
 $\mathfrak{X}_{k}$ is the sigma algebra generated by $x_0, \dots, x_{k}$, and  
 $P_{\xi_k\vert \mathfrak{X}_{k}}\colon \mathcal{I}\times\Omega \to \R_+$ denotes the conditional 
 distribution of $\xi_k$, given $\mathfrak{X}_{k}$. Finally, 
 \begin{equation}
 \label{eq:20200923a}
 (\forall\, \omega \in \Omega)\quad
 P_{\xi_k\vert \mathfrak{X}_{k}}(\cdot,\omega) 
= \nu(\cdot, x_{k}(\omega)).
 \end{equation}
\end{enumerate}

\begin{remark}
\normalfont
The above set control scheme subsume the knowledge of the Bregman distances to all $C_i$'s
from the current realization $x_k(\omega)$ of the $k$-th iterate and the possibility to compute the expectation 
$\EE[D_{C_{\xi}}(x_k(\omega))]$. This is feasible and can be efficiently implemented 
in the orthogonal sketch \& project method \cite{GMMN2020}. 
This is also the case for KL-projections in entropic regularized optimal transport,
however, we postpone a proof of this fact in a subsequent work. 
\end{remark}

\begin{remark}
\normalfont
Let $k \in \N$ and $\hat{x}_0, \dots, \hat{x}_{k-1}, x \in \inte(\dom \Leg)$.
One can define the conditional probability of $\xi_k$ given $x_0 = \hat{x}_0, \dots, x_{k-1} = \hat{x}_{k-1}, x_{k} = x$
\begin{equation}
P_{\xi_k\vert x_0 = \hat{x}_0, \dots, x_{k-1} = \hat{x}_{k-1}, x_{k} = x} \colon \mathcal{I} \to \R_+
\end{equation}
Then \eqref{eq:20200923a} is equivalent to ask for
\begin{equation}
(\forall\,x \in \inte(\dom \Leg))\quad P_{\xi_k\vert x_0 = \hat{x}_0, \dots, x_{k-1} = \hat{x}_{k-1}, x_{k} = x} = \nu(\cdot, x),
\end{equation}
which also implies $P_{\xi_k\vert x_0 = \hat{x}_0, \dots, x_{k-1} = \hat{x}_{k-1}, x_{k} = x} = P_{\xi_k\vert x_{k} = x}$
and does not depend on $k$.
\end{remark}

\begin{theorem}[Adaptive random set control]
\label{thm:main_adaptive}
Under the assumptions of Problem~\ref{P2},
suppose in addition that \ref{eq:stand1}, \ref{eq:AS3}, and \ref{eq:AS3b} hold. 
Let $V\colon\inte(\dom \Leg)\to \R_+$ and $\beta\colon\inte(\dom \Leg)\to \R_+$ be such that
\begin{equation}
V(x) := 
\begin{cases}
\VV \bigg[ \dfrac{D_{C_{\xi}}(x)}{\overline{D}_C(x)}\bigg] &\text{if } x \notin C\\[2ex]
0&\text{if } x \in C,
\end{cases}
\quad\text{and}\quad
\beta(x) = 1 + V(x).
\end{equation} 
Let $(x_k)_{k \in \N}$ be generated by Algorithm~\ref{algoABP} using
the adaptive random set control scheme \ref{eq:C3}  and
let $x_\star = P_{C}(x_0)$.
Set $\beta_\infty :=\liminf_{k\to\infty}(\essinf\beta(x_k))\geq1$ and, for every $k \in \N$,
 $\tilde{\sigma}_{\C,\mu}(x_k) = \beta(x_k) \sigma_{\C,\mu}(x_{k}) + 1-\beta(x_k)$.
Then the following holds. 
\begin{enumerate}[label={\rm (\roman*)}] 
\item\label{thm:main_adaptive_00} 
$(\forall\, k \in \N)\ 0\leq \tilde{\sigma}_{\C,\mu}(x_k) \leq \sigma_{\C,\mu}(x_k)\ \PP$-a.s.
\item\label{thm:main_adaptive_0} $(\forall\, k \in \N)\ \EE[D_C(x_{k+1})\,\vert\,\mathfrak{X}_{k}] 
\leq \tilde{\sigma}_{\C,\mu}(x_k) D_C(x_{k})\ \PP$-a.s.
\item\label{thm:main_adaptive_i} $(\forall\, n \in \N)(\forall\, k \in \N)\ 
\EE[D_C(x_{k+1+n})]\leq (\esssup\sigma_{\C,\mu}(x_{n})) \EE[D_C(x_{n+k})]$
\item\label{thm:main_adaptive_iii} 
For every $\varepsilon>0$ such that $(1+\varepsilon)(1-\varepsilon)^{-1} \sigma_{\C\,\mu}(x_\star)<1$ 
and every $\alpha \in \left]0,1\right[$ there exists $\mathcal{A} \in \mathfrak{A}$  with
$\PP(\Omega\setminus \mathcal{A})\leq\alpha$ and $n\in\N$, such that, for all $k\geq n$, 
\begin{equation*}
\EE[D_C(x_{k+1})\,\vert\,\mathcal{A}] \leq \Big( \beta_\infty 
\sigma_{\C,\mu}(x_\star)+1- \frac{1+\varepsilon}{1-\varepsilon}\beta_\infty\Big) \, \EE[D_C(x_{k})\,\vert\,\mathcal{A}] 
\end{equation*}
\end{enumerate}
\end{theorem}

\begin{proof} 
\ref{thm:main_adaptive_00}-\ref{thm:main_adaptive_0}:
The second inequality in \ref{thm:main_adaptive_00} is immediate, taking into account that,
since $\sigma_{\C,\mu}(x_k) \leq 1$, $\tilde \sigma_{\C,\mu}(x_k)$ is decreasing in 
$\beta(x_k) \geq 1$.
In the following we prove the first inequality of \ref{thm:main_adaptive_00} and the inequality in 
\ref{thm:main_adaptive_0} at the same time.
Using Lemma~\ref{lm:aff_proj_1}, similarly to the proof of Theorem~\ref{thm3b}\ref{prop_randomaffine_ii}, 
one obtains
\begin{equation}
\label{eq:20201217a}
\EE[D_C(x_{k+1})\,\vert\, \mathfrak{X}_{k}] = D_C(x_{k}) - \EE[D_{C_{\xi_k}}(x_{k})\, \vert\,\mathfrak{X}_{k}].
\end{equation}
Now, using \eqref{eq:20200923a} and \eqref{eq:20200924a}, we have that, for every 
$\omega\in\Omega$, 
\begin{equation*}
\EE[D_{C_{\xi_k}}(x_{k})\,\vert\, \mathfrak{X}_{k}](\omega) 
= \int_{I}  D_{C_{i}}(x_{k}(\omega)) P_{\xi_k\vert \mathfrak{X}_{k}}(d i,\omega)
= \int_{I} D_{C_{i}}(x_{k}(\omega))  \nu(d i, x_{k}(\omega)).
\end{equation*}
If $x_{k}(\omega)\in C$, then, clearly, 
$\EE[D_{C_{\xi_k}}(x_{k})\,\vert\, \mathfrak{X}_{k}](\omega) = \overline{D}_C(x_{k}(\omega))=0$
and \ref{thm:main_adaptive_0} and the first of \ref{thm:main_adaptive_00} hold when evaluated at $\omega$.
Otherwise, we have that
\begin{align}
\nonumber\EE[D_{C_{\xi_k}}(x_{k})\,\vert\, \mathfrak{X}_{k}](\omega) &= \int_{I} 
\dfrac{[D_{C_{i}}(x_{k}(\omega))]^2}{\overline{D}_{C}(x_{k}(\omega))} \mu(d i)\\
\nonumber&=\overline{D}_{C}(x_{k}(\omega)) \int_{I} 
\left(\dfrac{D_{C_{i}}(x_{k}(\omega))}{\overline{D}_{C}(x_{k}(\omega))}\right)^2 \mu(d i)\\
\label{eq:20200924b}&=\overline{D}_{C}(x_{k}(\omega)) (V(x_{k}(\omega)) + 1),
\end{align}
where we used the fact that 
\begin{equation}
\VV\bigg[ \dfrac{D_{C_\xi}(x)}{\overline{D}_C(x)}\bigg] 
= \EE\bigg[ \bigg( \dfrac{D_{C_\xi}(x)}{\overline{D}_C(x)}\bigg)^2\bigg] -
\bigg( \EE\bigg[ \dfrac{D_{C_\xi}(x)}{\overline{D}_C(x)} \bigg] \bigg)^2 = 
\EE\bigg[ \bigg( \dfrac{D_{C_\xi}(x)}{\overline{D}_C(x)}\bigg)^2\bigg] - 1.
\end{equation}
Set, for the sake of brevity, $\beta_k = \beta(x_k(\omega))$.
Then, by the definition of $\beta$, \eqref{eq:20200924b} yields
 $\EE[D_{C_{\xi_k}}(x_{k})\,\vert\, \mathfrak{X}_{k}](\omega) 
 = \beta_k \overline{D}_{C}(x_{k})(\omega)$.
Thus, by \eqref{eq:20201217a}, we have $\EE[D_C(x_{k+1})\,\vert\, \mathfrak{X}_{k}](\omega) 
= D_C(x_{k}(\omega))- \beta_k\overline{D}_C(x_{k}(\omega)) 
= (1 - \beta_k)D_C(x_{k}(\omega)) + \beta_k
(D_C(x_{k}(\omega)) - \overline{D}_C(x_{k}(\omega)))
=  (1 - \beta_k)D_C(x_{k}(\omega))
+ \beta_k \EE[D_C(P_{C_\xi}(x_{k}))\,\vert\, \mathfrak{X}_{k}](\omega)
\leq [(1 - \beta_k) + \beta_k\sigma_{\C,\mu}(x_k(\omega))]D_C(x_{k}(\omega))$,
where we used the definition of $\sigma_{\C,\mu}$ in \eqref{eq:sigma_mu} and that $x_k(\omega) \in K(x_k(\omega))\setminus C$. So, the inequality in \ref{thm:main_adaptive_0} and the first of \ref{thm:main_adaptive_00} hold when evaluated at $\omega$.

\ref{thm:main_adaptive_i}:
It is sufficient to prove that $(D_C(x_k))_{k\in\N}$ is $\PP$-a.s.~decreasing
and that $P_C(x_k) = x_{\star}\ \PP$-a.s., for every $k \in \N$.
Indeed, once we have proved this, we can follow the same line of argument as in
the proof of items \ref{thm:main_random_i}-\ref{thm:main_random_iv} of
Theorem \ref{thm:main_random}. To that purpose, 
by Proposition \ref{prop:stochasticintersection2}\ref{prop:stochasticintersection2_ii} we have that there 
exists $\mu$-negligible set $J$ such that $C=\bigcap_{i\in I\!\setminus J}C_i$.
Let $k \in \N$, set
$N_k = \xi_k^{-1}(J)$, and let $\omega\in\Omega$. Then
$\PP(\xi_k^{-1}(J) \,\vert\, \mathfrak{X}_k) = P_{\xi_{k} \vert \mathfrak{X}_{k}}(J, x_k(\omega))$.
Hence, if $x_k(\omega) \in C$, then $\PP(\xi_k^{-1}(J) \,\vert\, \mathfrak{X}_k) = \mu(J) = 0$,
otherwise
\begin{equation*}
\PP( \{x \in C_{\xi_k}\} \,\vert\,  \mathfrak{X}_{k})(\omega) = 
\int_{J}   \dfrac{D_{C_{i}}(x_{k}(\omega))}
{\overline{D}_{C}(x_{k}(\omega))} \mu(d i) = 0.
\end{equation*}
Therefore in any case $\PP(N_k \,\vert\, \mathfrak{X}_k)=0$ and 
hence $\PP(N_k)= \EE[\PP(N_k \,\vert\, \mathfrak{X}_k)]=0$. 
Now, as in the proof of Proposition~\ref{prop:stochasticintersection2}\ref{prop:stochasticintersection2_ii},
we prove that $C \subset \bigcap_{\omega \in \Omega\setminus N_k} C_{\xi_k(\omega)}$.
Indeed, let $x \in C$. Then, for every $\omega \in \Omega\setminus N_k$, we have $\xi_k(\omega) \in I\setminus\! J$
and hence, since $x \in C = \bigcap_{i\in I\setminus J} C_i$, we have $x \in C_{\xi_k(\omega)}$.
Then, for every $\omega \in \Omega\!\setminus N_k$, 
$x_{k+1}(\omega) = \Pc_{C_{\xi_k(\omega)}}(x_{k}(\omega))$ and $C\subset C_{\xi_k(\omega)}$. Hence, using 
Lemma \ref{lm:aff_proj_1}, we have $P_C(x_{k+1}(\omega)) = P_C(x_k(\omega))$ and $D_C(x_{k+1}(\omega)) 
= D_C(x_{k}(\omega)) - \Bre(x_{k+1}(\omega),x_{k}(\omega))\leq D_C(x_{k}(\omega))$. This proves that 
$P_C(x_{k+1}) = P_C(x_k)$ $\PP$-a.s.~and
$D_C(x_{k+1}) \leq D_C(x_{k})$ $\PP$-a.s.

\ref{thm:main_adaptive_iii}:
Let $\varepsilon>0$ with $(1+\varepsilon)(1-\varepsilon)^{-1} \sigma_{\C\,\mu}(x_\star)<1$
and $\delta>0$ be as in Lemma~\ref{lm:sigma_mu}\ref{lm:sigma_mu_ib}. 
The proof follows that  of Theorem~\ref{thm:main_random}\ref{thm:main_random_iv}.
Set $\mathcal{A}_{k}:=\{ D_C(x_{k})\leq\delta \}$ and
let $n \in \N$ be sufficiently large so that $\mathcal{A}:=\mathcal{A}_n$ is such that
$\PP(\Omega\setminus\mathcal{A}) \leq \alpha$. Let $k \in \N$, with $k \geq n$.
Then $\mathcal{A} \subset \mathcal{A}_k$ and hence, 
 denoting by $\chi_{\mathcal{A}}$ the characteristic function of $\mathcal{A}$, 
we have
$\sigma_{\C,\mu}(x_k) \chi_{\mathcal{A}}  \leq (1 + \varepsilon)(1-\varepsilon)^{-1}\sigma_{\C,\mu}(x_\star)$.
Next, by  \ref{thm:main_adaptive_0} and the definition of $\tilde{\sigma}_{\C,\mu}(x_k)$, 
\begin{align*}
\EE[D_C(x_{k+1})\,\vert\,\mathfrak{X}_{k}] \chi_{\mathcal{A}} 
& \leq  \Big[1 -\beta(x_k) \Big( 1-  \frac{1+\varepsilon}{1-\varepsilon} \sigma_{\C,\mu}(x_\star) \Big) \Big] 
D_C(x_{k}) \chi_{\mathcal{A}}\\
& \leq  \Big[1 - \inf_{h \geq n}(\essinf \beta(x_h)) \Big( 1-  \frac{1+\varepsilon}{1-\varepsilon} \sigma_{\C,\mu}(x_\star) \Big) \Big] 
D_C(x_{k}) \chi_{\mathcal{A}}.
\end{align*}
Now let $n \in \N$ be sufficiently large so that 
$\inf_{h \geq n}(\essinf \beta(x_h)) > (1-\varepsilon)(1+\varepsilon)^{-1} \beta_\infty$. 
Then the statement follows after taking expectation.
\end{proof}

\begin{remark}
\normalfont
Theorem~\ref{thm:main_adaptive} shows the advantage, in terms of convergence rate,
of the adaptive random set control strategy \ref{eq:C3} against \ref{eq:C2}.
In particular, in the inequality in Theorem~\ref{thm:main_adaptive}\ref{thm:main_adaptive_iii}, if
$\beta_\infty >1$ one can choose $\varepsilon>0$ sufficiently small and get
a better rate than that of Theorem~\ref{thm:main_random}\ref{thm:main_random_iv}.
\end{remark}

\section{Applications}
\label{sec:Applications}

In this section we present two applications, where our analysis provides novel results and/or possibilities. 
We provide only general insights, leaving a deeper treatment of the subjects for future work. 
In Table~\ref{tab:funs} we provide several examples of Legendre functions satisfying the assumptions of our convergence theorems
that can be used in such applications.

\begin{table}[t]
\begin{center}\small
\caption{Legendre functions $\Leg(x)=\sum_{i=1}^{n}\varphi(x_i)$ with $\Leg^*$ twice differentiable and $\dom\Leg^*$ open.}
\label{tab:funs}
\vspace{1ex}
\begin{tabular} {|>{\centering}m{0.16\textwidth}|>{\centering}m{0.24\textwidth}|>{\centering}m{0.06\textwidth}|>{\centering}m{0.12\textwidth}|>{\centering\arraybackslash}m{0.21\textwidth}|}
\hline
Legendre function & $\varphi(t)$ & $\dom \varphi$ & $\dom \varphi^*$ & Bregman distance/ Application \\ \hline \hline
Burg entropy & $-\log t$ \phantom{\Big(} & $\R_{++}$ & $\R_{--}$  & Itakura-Saito divergence \cite{DPR2018}\\ \hline
Boltzmann-Shannon entropy & $t\log t-t$& $\R_+$ & $\R$  & Kullback-Leibler divergence \cite{BEN2015,PC2019} \\ \hline
Fermi-Dirac entropy & $t \log t + (1-t) \log(1-t)$ & $[0,1]$ & $\R$  & Logistic loss \cite{HL2005,DPR2018}\phantom{\Big(} \\ \hline
Hellinger entropy & $-\sqrt{1-t^2}$ & $[-1,1]$ & $\R$  & Hellinger distance \cite{DPR2018}\phantom{\Big(} \\ \hline
Positive power $0<\beta< 1$ & $\dfrac{t^{\beta} - \beta t + \beta -1}{\beta(\beta-1)}$ & 
$\R_+$ & $\left]-\infty, \frac{1}{1-\beta}\right[$ & 
$\beta$-divergence \cite{DPR2018,FI2011}\\ \hline
Tsallis entropy $0<q<1$ & $\frac{1}{q-1}(t^{q}-t)$ & $\R_{+}$ & $\left]-\infty, \frac{1}{1-q}\right[$  & Tsallis relative $q$-entropy \cite{MUZ2017} \\ \hline
$p$-norm $1<p\leq 2$ & \phantom{\Big(}$\tfrac{1}{p} \abs{t}^p$\phantom{\Big)} & $\R$ & $\R$ &   Compressed sensing \cite{AH2019}\\ \hline
\end{tabular}
\end{center}
\end{table}%

\subsection{Sketch \& Bregman project methods}
\label{sec:SketchProject}

In recent years, a new class of iterative solvers for linear systems, 
called \emph{sketch \& project} \cite{GR2015}, has emerged, taking its origin in the work by 
Strohmer and Vershynin \cite{SV2009}, which, in turn, generalized the Kaczmarz method \cite{KAC1937}. 
Below we describe the method. Given the problem
\begin{equation*}
x_\star = \argmin\tfrac{1}{2}\scalarp{B x,x}\quad\text{subject to}\quad A x = b, 
\end{equation*} 
where $A\in\R^{m\times n}$, $b\in\R^m$, and $B\in\R^{n\times n}$ is a symmetric positive definite matrix, 
one introduces a family of \emph{sketch} matrices $(S_i)_{i\in I}$ such that $S_i\in\R^{m\times n_i}$, $n_i\geq1$, $i\in I$, and then replaces the task of solving $A x=b$ 
by that of projecting $B$-orthogonally onto the solution set of the \emph{sketched} systems $S_i^*A x = S_i^*b$, $i\in I$,
according to different types of sketching control schemes (cyclic, greedy, random, adaptive) \cite{GMMN2020}.
This method indeed leads to the computation of the minimal $B$-norm solution of $A x=b$ under the assumption that $x_0=0$ and that the sketches are chosen in a way that the whole solution space is covered 
by the sketching process. Denoting the solution sets by
\begin{equation*}
C = \{x \in \R^n \,\vert\, A x = b\}\,\text{ and } C_i = \{x \in \R^{n} \,\vert\, S_i^*A x = S_i^*b\}\,(i\in I),
\end{equation*} 
the last assumption means that $C=\bigcap_{i\in I}C_i$, for a deterministic sketching strategy, or
$C=\{x\in\R^n\,\vert\,x\in C_{\xi}\ \PP\text{a.s.}\}$,
when the sketches are chosen randomly. This last condition
was referred to as the \emph{exactness assumption}, and can be equivalently expressed as $\range(A)\cap \Ker(\EE[S_\xi(S_\xi^* AA^*S_\xi)^\dagger S_\xi^*])=\{0\}$ \cite[Theorem 3.5]{RT2020}.
Global linear convergence rates were obtained in \cite{GMMN2020,SV2009,RT2020}, which extend those related to the Kaczmarz-type methods.

A direct application of our results is that the existing sketch \& project methodology can be successfully extended to Bregman projections, leading to the novel archetypal method of \emph{sketch \& Bregman project}.  In such a way, instead of finding the minimal norm solution of the linear system as in the classical setting, we compute the solution that minimizes the Legendre function $\Leg$ over the solutions of $A x = b$, provided that
$\nabla\Leg(x_0)=0$. In Table \ref{tab:funs} we give some typical separable Legendre functions that can be used within our framework together with references. 

The most studied family of sketches is the one associated to the standard basis of $\R^m$, e.g., 
$S_i = e_i$, which produces the popular Kaczmarz (or randomized Kaczmarz) method. In our setting this leads to \emph{Bregman-Kaczmaz methods} (greedy, random, adaptive random). In this specific case, 
 the expression of the rates $\sigma_{\C}$ and $\sigma_{\C, \mu}$ can be simplified. 
 For instance, let  $\mu=1_m/m$ be the probability vector of the uniform distribution on $I=\{1,\ldots,m\}$. 
 Then the constants in \eqref{eq:sigma_I} and \eqref{eq:sigma_mu} become
\begin{equation}\label{sbp_sigmas}
\sigma_{\C}(x_\star) = 1-\min_{v\in\range(\bar A^*)}\frac{\norm{\bar A v}_{\infty}}{\norm{v}_2} \;\text{ and }\;\sigma_{\C,\mu}(x_\star) = 1-\tfrac{1}{m}\lambda_{\min}^{+} (\bar A^*\bar A),
\end{equation}
respectively, where, setting
$D = \Diag(\lVert [\nabla^2\Leg(x_\star)]^{-1/2}A_{1:}\rVert_2,\dots,\lVert [\nabla^2\Leg(x_\star)]^{-1/2}A_{m:}\rVert_2)$, we have $\bar A =D^{-1}A$. 
These results, when specialized to $\Leg=(1/2)\norm{\cdot}^2_2$,  recover the ones of the Kaczmarz method \cite{KAC1937,SV2009}.

We finish the section by addressing a comparison with \cite{GMMN2020}. 
When 
$\Leg =(1/2) \norm{\cdot}_{B}^2$ (so that $\Bre(x,y)= (1/2)\norm{x-y}_{B}^2$), 
 Fact \ref{prop:bregman_div}\ref{prop:bregman_div_vii} holds for an arbitrarily large $\delta$ and hence, 
the local linear rates given in
 Theorems \ref{thm:main2}\ref{thm:main2_iii}, \ref{thm:main_random}\ref{thm:main_random_iii} and \ref{thm:main_adaptive}\ref{thm:main_adaptive_iii}, are indeed global and match the ones in \cite{GMMN2020}.
 Moreover we proved $\PP$-a.s.~convergence of the iterates, which is, up to our best knowledge, a new result even in the classical setting of stochastic sketch \& project methods.

\subsection{Multimarginal regularized optimal transport}
\label{sec:ROT}

The role of the Bregman projection method in optimal transport (OT) is well-known and deeply studied.
In this section we describe the more general setting of multimarginal OT.
In particular, we consider the discrete multimarginal regularized optimal transport problem as described in \cite{BEN2015}.
Let $n_1, \dots, n_m\in\N\!\setminus\!\{0,1\}$.
Let, for every $i=1, \dots, m$, $\rho_i\in\Delta_{n_i}$, where $\Delta_\ell = \{x\in\R_+^\ell \,\vert\, \norm{x}_1 = 1\}$ is the unit simplex of $\R^\ell$. Set $X = \R^{n_1\times \cdots \times n_m}$ and
 define, for every $i=1, \dots, m$, the push-forward (linear) operator $A_i \colon X \to \R^{n_i}$
 \begin{equation}
 A_i \pi = \Big( \sum_{h_1=1}^{n_1} \dots \sum_{h_{i-1}=1}^{n_{i-1}} \sum_{h_{i+1}=1}^{n_{i+1}} \cdots \sum_{h_{m}=1}^{n_{m}} \pi_{h_1, \dots, h_{i-1}, h, h_{i+1}, \cdots, h_{m}}\Big)_{1 \leq h \leq n_i}. 
 \end{equation}
The objective of multimarginal (entropic) regularized OT can be formulated as
 \begin{equation}
 \label{RMOTproblem}
 \pi_\star=\argmin_{\pi\in C_{1} \cap\dots\cap C_{m}} \mathrm{KL}(\pi,\kappa),
 \quad C_i = \{\pi \in X \,\vert\, A_i \pi = \rho_i\},\ i=1,\dots, m,
 \end{equation}
where 
$\kappa = e^{- \alpha/\eta} \in X$ is the \emph{Gibbs kernel}, $\alpha \in X$ the \emph{cost multidimensional matrix},
and $\mathrm{KL}(\pi,\kappa)$ is the Kullback-Leibler divergence, defined as
\begin{equation}
\mathrm{KL}(\pi,\kappa) = \sum_{h_1, \dots, h_m} \pi_{h_1, \dots, h_m} \log\Big(\frac{\pi_{h_1, \dots, h_m}}{\kappa_{h_1, \dots, h_m}} \Big) - \pi_{h_1, \dots, h_m} + \kappa_{h_1, \dots, h_m},
\end{equation}
which is the Bregman distance associated to the Boltzmann-Shannon entropy 
\begin{equation*}
\phi(\pi) = \sum_{h_1}^{n_1}\dots\sum_{h_m}^{n_m} \pi_{h_1,\dots, h_m}( \log \pi_{h_1,\dots,h_m} -1).
\end{equation*}
See Table~\ref{tab:funs}.

Since the KL projection onto the affine sets $C_i$'s can be computed explicitly and 
efficiently \cite[Proposition~4]{BEN2015}, our results show that 
the Bregman projection algorithm with greedy or random set control
sequence applied to problem \eqref{RMOTproblem}, and initialized with $\pi_0= \kappa$, 
converges Q-linearly, in the KL distance, to $\pi_\star$.
In this setting, when $m=2$, the Bregman projection method with greedy set control scheme
reduces to the Sinkhorn algorithm \cite{PC2019} (see Remark~\ref{rmk:greedy-alternate}),
for which both global and local linear rates of convergence are known \cite{KNI2008,PC2019}. However,
for $m>2$ the above results are new. Finally, note that the linear operator
\begin{equation}
A\colon X \to \R^{n_1}\times\cdots\times\R^{n_m},\quad
A \pi = (A_i \pi)_{1 \leq i \leq m},
\end{equation}
can be identified with a (possibly huge) matrix of dimensions $(n_1+\cdots + n_m)\times (n_1\cdots n_m)$.
So, we could even consider to perform Bregman projections onto each single row of the linear system
\begin{equation}
\label{eq:OTsys}
A \pi = 
\begin{bmatrix}
\rho_1\\
\vdots\\
\rho_m
\end{bmatrix}.
\end{equation}
This approach, if implemented with a greedy set control sequence, leads
to the Greenkhorn algorithm which was studied in \cite{AWR2017} only for the standard bimarginal OT problem.
We stress that, to the best of our knowledge, the Q-linear rate of convergence is new for this 
algorithm too (even when $m=2$).

\section{Conclusion and future work}

We studied the Bregman projection method with greedy and random set control sequences in deterministic 
and stochastic convex feasibility problems. We focused in particular on affine feasibility problems showing 
the Q-linear convergence of the method. In our future work on this subject we will devote special attention 
to applications  in regularized (multimarginal) OT, which here we just touched. Among other issues, we will 
address the explicit computation of the local Q-linear rate in terms of the problem data, the computation of 
the rate for Wasserstein barycenters formulated as a multimarginal OT, and the effects of different sketching 
control schemes on the convergence rate.

\end{document}